\documentclass{amsart}

\usepackage{amsfonts,amsmath,amssymb}
\usepackage{bm}

\usepackage{amsrefs}
\usepackage{mathrsfs}

\usepackage{enumerate,graphicx}

\usepackage[colorlinks=true,breaklinks=true,bookmarks=true,urlcolor=blue,
citecolor=black,linkcolor=black,bookmarksopen=false,draft=false]{hyperref}

\newtheorem{theorem}{Theorem}[section]
\newtheorem{lemma}[theorem]{Lemma}

\newtheorem*{assumption*}{Assumption}

\theoremstyle{definition}
\newtheorem{definition}[theorem]{Definition}

\theoremstyle{remark}
\newtheorem{remark}[theorem]{Remark}

\numberwithin{equation}{section}

\title[small noise perturbation for ergodic  control]{Small noise  perturbations of
	stochastic ergodic  control problems}
\date{\today}

\begin{document}

	\author[Suresh]{K. Suresh Kumar}
	\address[Suresh]{Department of Mathematics, Indian Institute of Technology Bombay, 
		Mumbai 400076, India.}
	\email{suresh@math.iitb.ac.in}

	\author[Vikrant]{Desai, Vikrant}
	\address[Vikrant]{Department of Mathematics, Indian Institute of Technology Bombay, 
		Mumbai 400076, India.}
	\email{vikrant@math.iitb.ac.in}

	\maketitle

\begin{abstract}
Using small noise limit approach, we study degenerate stochastic ergodic control problems
and as a byproduct obtain error bounds for the $\varepsilon$-optimal controls. We also 
establish tunneling for a special ergodic control problem and give a representation of
the ergodic value using the tunneled Markov chain.
\end{abstract}

\section{Introduction}
\label{sec:intro}
We study degenerate stochastic control problems using "small noise" perturbation analysis.
i.e.,  we  approach the original as a  limit of non-degenerate controlled diffusions given
by  small noise perturbations. This is in the spirit of a proposal of Kolmogorov as cited in
\cite{EckmannRuelle} to select a 'physical solution' for an ill-posed problem by
looking at the small noise limit of its stochastic perturbation. This is also the
philosophy behind the vanishing viscosity method for solving elliptic/parabolic pdes.  
Since we are concerned with ergodic control 
problems, the behaviour of the invariant probability measures of the state dynamics enter the 
picture. Naturally one  expects that the asymptotic behaviour of the invariant measure of the 
small noise perturbation, as perturbation noise approach zero, play a role in the 
small noise limit of the ergodic control problem. Small noise perturbation 
of dynamical systems governed by odes and the asymptotics of the invariant probability measures
is studied in Chapter 6, \cite{FreidlinWentzell} when  state space is compact and in 
\cite{AnupBorkar} when  state space is $\mathbb{R}^d$. Small noise limit analysis leads to a selection procedure of invariant measure for the limiting dynamics which may possess multiple
invariant  probability measures, see \cite{MattinglyPardoux}.  Analogous to the selection procedure for  the invariant measures, does small noise perturbation of the ergodic optimal control problem  selects in the limit  a unique  optimal ergodic control problem for the original degenerate controlled state dynamics We get an affirmative answer under suitable conditions.  We call this as  'physical' ergodic control problem.
The selection is  achieved through the selection of an invariant probability measure from the
multiple invariant probability measures of the optimal state dynamics and hence our result 
also can be  thought as a controlled version of the selection procedure described in  
\cite{FreidlinWentzell}.
For a special case, we  see a tunneling of the optimal state dynamics and arrive at a
representation of the ergodic value through the tunneled Markov chain.

We consider  degenerate stochastic control problem
with state dynamics  governed  by  the controlled degenerate stochastic differential equation 
\begin{equation}\label{statedegensde}
	d X(t) \  = \ m(X(t), U(t)) dt + \sigma (X(t) ) d W(t).
\end{equation}
The functions  $m : \mathbb{R}^d \times {\mathcal P} (\mathbb{U}) \to \mathbb{R}^d,
\sigma : \mathbb{R}^d \to \mathbb{R}^d \otimes \mathbb{R}^d$ are  bounded measurable, with 
${\mathcal P} (\mathbb{U})$ denoting the space of probability measures on a compact metric space 
$\mathbb{U}$ endowed with the Prohorov topology. Note that ${\mathcal P} (\mathbb{U})$ is a compact
Polish space\footnote{ More generally, we shall
	denote by ${\mathcal P}(\cdots)$ the Polish space of probability measures on the
	Polish space `$\cdots$' with Prohorov topology.}. 
$W(\cdot)$ is a $d$-dimensional Wiener process and 
$U(\cdot)$ is a $\mathcal{P}(\mathbb{U})$-valued
process which is progressively measurable with respect to a right-continuous filtration 
$\{ {\mathcal F}_t \}$ which
is complete with respect to the underlying probability measure, and furthermore, is
non-anticipative with respect to  $W(\cdot)$, i.e., for $t > s \geq 0$, $W(t) - W(s)$ is
independent of ${\mathcal F}_s.$ We denote the set of admissible controls by
${\mathcal U}$.

Furthermore,  we  assume that $m$ is  of the form
\begin{equation}
	m(x, U) \ = \ \int_{\mathbb{U}}  \Bar{m}(x, u) U(du), x \in \mathbb{R}^d, U \in 
	\mathcal{P}(\mathbb{U}) \label{relax1}
\end{equation}
where $\Bar{m}: \mathbb{R}^d \times \mathbb{U} \mapsto \mathbb{R}^d$ is locally bounded and  measurable. 


The cost criteria is given by
\begin{equation}\label{ergodiccost}
	\rho(x, U(\cdot)) \ = \ \liminf_{t \to \infty} \frac{1}{t} E_x \Big[ \int^t_0 r(X(s), U(s)) ds
	\Big] , \ U (\cdot) \in {\mathcal U}
\end{equation}
where the prescribed `running cost function' 
$r: \mathbb{R}^d\times {\mathcal P} (\mathbb{U}) \mapsto \mathbb{R}$ is of the form
\begin{equation}
	r(x, U) \ = \ \int_{\mathbb{U}}  \Bar{r}(x, u) U(du) \label{relax1.5}
\end{equation}
and $E_x [\cdot]$ denote the expectation with respect to the law of the process 
$(X(\cdot), U(\cdot))$ given by (\ref{statedegensde}) with initial condition 
$X(0) = x \in \mathbb{R}^d$. 
We set $\rho(x, U(\cdot)) = \infty$ if corresponding to $U(\cdot) \in {\mathcal U}, $ the stochastic differential equation (in short sde) 
(\ref{statedegensde}) has no solution with initial condition $x$.

The structural assumptions (\ref{relax1}), (\ref{relax1.5})  on
$m, r$ given above is a part of  the so called \textit{relaxed control formulation}
introduced by L.\ C.\ Young \cite{Young}, which in particular ensures the existence
of an optimal control under fairly general conditions.

We say that an admissible control  $U(\cdot)$ is a Markov control if 
$U(t) = u(t, X (t)), t \geq 0$ for some measurable map 
$u : [0, \  \infty) \times \mathbb{R}^d \to \mathcal{P}(\mathbb{U})$, where $X(\cdot)$ 
is a  weak solution to  (\ref{statedegensde}) corresponding to  $u (\cdot)$.  
We denote the set of all Markov controls by 
${\mathcal U}_M$ and by an abuse of notation we take 

$
{\mathcal U}_M = \{ u : [0, \  \infty)  \times  \mathbb{R}^d \to \mathcal{P}(\mathbb{U}) \,  
|  \,  u \  {\rm is\ measurable,\ sde \ (\ref{statedegensde})\  has\ a\ weak\ solution}$ 

$ \  \  \  \ \ \  \  \  \  \  \  \  \  \ {\rm corresponding\ to} \ u(\cdot) \}. $ 

When a Markov control $u (\cdot)$ doesn't has an explicit depedendence on $t$, we call it as 
a stationary Markov control and the set of all stationary Markov controls is denoted by
${\mathcal U}_{SM}$. Set

\[\hat {\mathcal U}_{SM} \ = \ \{  u :   \mathbb{R}^d \to \mathcal{P}(\mathbb{U}) \,  
|  \,  u \  {\rm is\ measurable} \}.
\]
We use the following topology on $\hat {\mathcal U}_{SM}$ described in
\cite{AriBorkarGhosh}, p.57. Consider 
$L^\infty (\mathbb{R}^d ; \mathcal{M}_s (\mathbb{U}))$ endowed with weak*-topology, where 
$\mathcal{M}_s (\mathbb{U})$ denote the space of signed Borel measures  on $\mathbb{U}$
endowed with weak* topology. Then $\hat {\mathcal U}_{SM}$ is a subset of the unit ball in 
$L^\infty (\mathbb{R}^d ; \mathcal{M}_s (\mathbb{U}))$ and hence is given the relative topology.
Under this weak* topology $\hat {\mathcal U}_{SM}$ is compact due to Banach-Algaoglu theorem 
and is metrizable and the topolgy can be characterized by the following convergence criterion.
$u_n \to u$ in $\hat {\mathcal U}_{SM}$ if  for 
$ f \in L^1(\mathbb{R}^d) \cap  L^2(\mathbb{R}^d),  g \in C_b (\mathbb{R}^d \times \mathbb{U})$, 
\begin{equation}\label{BorkarTopology}
	\lim_{n \to \infty} \int_{\mathbb{R}^d}  f(x) \int_{\mathbb{U}} g(x, v) u_n (x) (dv) dx =
	\int_{\mathbb{R}^d}  f(x) \int_{\mathbb{U}} g(x, v) u (x) (dv) dx .
\end{equation}
For ${\mathcal U}_{SM}$, we use the relative topology described above. When $\sigma =0$, 
we can see that $\hat {\mathcal U}_{SM} = {\mathcal U}_{SM}$.

\noindent We also use smooth controls defined in the following sense. $u \in {\mathcal U}_{SM}$
is said to be smooth if $x \mapsto \int_{\mathbb{ U}} f(v) u(x)(dv) $ is smooth for all $f \in C(\mathbb{U})$.
We denote the set of all smooth controls by ${\mathcal U}^{smooth}_{SM}$.

We can  view any admissible control  $U(\cdot)$  as  a ${\mathcal U}_D$-valued random variable, where $ {\mathcal U}_D \ = \ \{ u : [0, \ \infty) \to \mathcal{P} (\mathbb{U})
| \ u \ {\rm is\ measurable} \}$. 
We then say that $U_n (\cdot) \to U(\cdot)$ in law if the
law of $U_n(\cdot)$ converges to the law of $U(\cdot)$ in $\mathcal{P} ({\mathcal U}_D)$.
By Prohorov's theorem,  $\mathcal{ P}({\mathcal U}_D)$ is a compact Polish space.

We can also view $U(\cdot) \in {\mathcal U}$ as a prescribed control in the sense that, 
one defines $U(\cdot)$ on a given $(\Omega, {\mathcal F},  \{{\mathcal F}_t \}, P,  W(\cdot))$. 
For instance, given  a feedback  control $u(\cdot)$, if $(X(\cdot), W(\cdot)) $ is a weak
solution  to (\ref{statedegensde}) defined on a filtered probability space  
$(\Omega, {\mathcal F}, \{{\mathcal F}_t\}, P) $ and 
$\{ {\mathcal F}_t \}$ - Wiener process  $W(\cdot)$, then the  process $U(t) = u(t, X_{[0, t]})$, 
where $X_{[0, t]} = \{ X(s) | 0 \leq s \leq t \}$ is a   
prescribed control on $(\Omega, {\mathcal F},  \{{\mathcal F}_t \}, P,  W(\cdot))$. This way we treat any feedback  control as a  prescribed control.

Before proceeding  further,  we  briefly list  notations used in the article. 
For the Euclidean space $\mathbb{R}^d,  x \in \mathbb{R}^d, \ x_i$ denotes its $i$th component
and $\| \cdot\|$ denotes 
$$\|x \| = \sqrt{x^2_1 + \cdots + x^2_d}, \ x \in \mathbb{R}^d,$$
$x \in \mathbb{R}^d$ is treated as a row vector.
For $A \in \mathbb{R}^d  \otimes \mathbb{R}^d $, the set of all $d \times d$ real
matrices, we use
$\|A \| = \sqrt{\sum^d_{i, j =1} a^2_{ij}},  A = (a_{ij}), \, A^T$ denotes the
transpose of $A$ and 
$I \in \mathbb{R}^d \otimes \mathbb{R}^d$ denotes the identity matrix.
$ C^k(\mathbb{R}^d), k=0,1,2$, denote the space of all
functions with continuous partial derivatives of order upto $k$.  For a function 
$f \in C^1 (\mathbb{R}^d)$, $\nabla f$ denotes the gradient and for a function $f
\in C^2(\mathbb{R}^d),
\  \nabla^2 f$ denotes the Hessian of $f$ and $\triangle f$ denotes its Laplacian.
$C^\infty_c(\mathbb{R}^d)$ denote the space of all smooth compactly supported functions defined
on $\mathbb{R}^d$. 
For the compact  metric space $\mathbb{U}, \ C(\mathbb{U})$ denotes the space of all continuous functions $f : \mathbb{U} \to \mathbb{R}$
with sup norm $\| f\|_\infty = \sup_{u \in \mathbb{U} } |f(u)|$.  Also we denote the set of all 
Lipschitz continuous function on $\mathbb{R}^d$ by $C^{0, 1}(\mathbb{R}^d)$. For $f \in 
C^{0,1}(\mathbb{R}^d)$, we denote its Lipschitz constant by Lip$(f)$.  We also use  
the Sobolev space $W^{2, p} (\mathbb{R}), p \geq 2$, the set of all $f \in L^p (\mathbb{R})$ 
such that $f' \in L^p(\mathbb{R})$ with the $p$-Sobolev norm. 

We use $K > 0$ to denote  the constant appearing in (A1)(ii).  Other constants are
denoted by $ \hat{K}, K_0,  K_1, K_2 $ etc., their values will
change from place to place depending on the context.

Controlled diffusion process $\{X(t) | 0 \leq t < \infty \}$ 
will also be  denoted by $X$ or $  X(\cdot) $. We use
the following parametric family of infinitesimal generators.
\begin{eqnarray}\label{generators}
	{\mathcal L}^U f (x) & = & \frac{1}{2}  {\rm trace} (a (x) \nabla^2 f)
	+ \langle m(x, U) , \nabla f \rangle, \\ \nonumber 
	{\mathcal L}^U_\varepsilon f (x) & = & \frac{1}{2}  {\rm trace} (a_\varepsilon (x)
	\nabla^2 f)
	+ \langle m(x, U) , \nabla f \rangle, 
	\nonumber
\end{eqnarray}
where $a = \sigma \sigma^T, a_\varepsilon = \sigma_\varepsilon \sigma^T_\varepsilon,
\ \sigma_\varepsilon$ is a suitable non degenerate
perturbation of $\sigma$ defined in forthcoming sections.


We say that $(X(\cdot), U(\cdot))$ is an admissible pair if $U(\cdot) \in {\mathcal U}$ and
$X(\cdot)$ is a weak solution to the sde (\ref{statedegensde}) corresponding to 
$U(\cdot)$. 
\begin{definition}\label{ergodicoptimal1}
	(i) Admissible control $U^*(\cdot)$ is ergodic optimal if
	\begin{equation}
		\rho(x, U^*(\cdot)) \ \leq \ \rho (x, U(\cdot)), \ {\rm for\ all} \ U(\cdot) \in 
		{\mathcal U}, \ x \in \mathbb{R}^d. 
	\end{equation}
	(ii) Admissible pair $(X^*(\cdot), U^*(\cdot))$ is an ergodic optimal pair if 
	\[
	\liminf_{t \to \infty} \frac{1}{t}  E \Big[ \int^t_0 r(X^*(s), U^*(s)) ds \Big]
	\leq \ \liminf_{t \to \infty} \frac{1}{t}  E \Big[ \int^t_0 r(X(s), U(s)) ds \Big]
	\]
	for all admissible pair $(X(\cdot), U(\cdot))$ of (\ref{statedegensde}).
\end{definition}
Let ${\mathcal G}$ denote the set of all ergodic
occupation measures of the process (\ref{statedegensde}), i.e. 
\[
{\mathcal G} \ = \ \Big\{ \pi \in \mathcal{P}(\mathbb{R}^d \times \mathbb{U}) \Big| 
\iint {\mathcal L}^U f(x, u) \pi(dx du ) \ = \ 0 ,  \ {\rm for\ all} \ f \in 
C^\infty_c (\mathbb{R}^d) \Big\},
\]
see Chpater 6 of \cite{AriBorkarGhosh} for details. We introduce various notions of values for
ergodic optimal control. 
\begin{eqnarray}\label{ergodicvalues}
	\rho^{**} & = & \inf_{\pi \in {\mathcal G}} \iint \bar r(x, u) \pi(dx du), \, 
	\rho^* \  = \  \inf_{(X(\cdot), U(\cdot)) } \liminf_{t \to \infty} \frac{1}{t} 
	E \Big[ \int^t_0 r(X(s), U(s)) ds \Big] \\ \nonumber 
	\rho^*_{\mathcal U} (x) & = & \inf_{U(\cdot) \in {\mathcal U}} \rho(x, U(\cdot)), \  x \in \mathbb{R}^d, \ \rho^*_{\mathcal U} \ = \ \inf_{x \in \mathbb{R}^d} \rho^*_{\mathcal U} (x),
	\\ \nonumber 
	\rho^*_{{\mathcal U}_{SM}} (x) & = & \inf_{U(\cdot) \in {\mathcal U}_{SM}} 
	\rho(x, U(\cdot)), \  x \in \mathbb{R}^d, \ \rho^*_{{\mathcal U}_{SM}} 
	\ = \ \inf_{x \in \mathbb{R}^d} \rho^*_{\mathcal U} (x), 
\end{eqnarray} 
where infimum in the second equality is over all admissible solution  pair $(X(\cdot), U(\cdot))$. Clearly $\rho^* \leq \rho^{**} \leq \rho^*_{\mathcal U} \leq 
\rho^*_{{\mathcal U}_{SM}}$. Under (A1) and standard stability assumptions, all the above 
ergodic optimal values coincide when (\ref{statedegensde}) satisfies the non degeneracy
condition but that may not be the case with out the non degeneracy condition.

We assume:
\begin{assumption*}[\bf A1]
	
	\begin{itemize}
		\item The functions
		$\bar m = (\bar m_1, \cdots ,  \bar m_d), \, \bar r $ are  jointly continuous and Lipschitz in
		the first variable uniformly with respect to the second, 
		$\sigma$ is Lipschitz continuous and bounded and $\sigma \sigma^T \geq 0$.
		
		\item There exists $K \in \mathbb{R}$ such that 
		\[
		\langle \bar m(x, u) - \bar m(y,u), \ x-y \rangle +  \frac{1}{2} \| \sigma (x) -
		\sigma (y) \|^2 \ \leq \ K \|x-y \|^2, \  x, y, \in \mathbb{R}^d, u \in \mathbb{U}.
		\]
		
	\end{itemize}
\end{assumption*}

\vspace{.1in}

Now let us briefly indicate what is known about ergodic control problems when the 
state dynamics are degenerate, i.e. when the diffusion matrix doesn't satisfies the 
uniform ellipticity condition. Following  general result is known, see 
[\cite{AriBorkarGhosh}, Chapter 7, Theorem 7.2.1, p.254].
\begin{theorem}\label{BorkarErgodicresult1} Assume (A1) and that the map 
	$(x, U) \mapsto {\mathcal L}^U V (x)$ is inf-compact for some non negative inf-compact 
	$V \in C^2(\mathbb{R}^d)$. Then there exists an optimal ergodic pair $(X(\cdot), U(\cdot))$ 
	such that $X(\cdot)$ is a Markov and $U(\cdot)$ can be taken as a stationary Markov control. 
\end{theorem}
Hence Theorem \ref{BorkarErgodicresult1}
doesn't give an optimal control in the sense of  Definition \ref{ergodicoptimal1} (i)
but give an existence result in the sense of Definition \ref{ergodicoptimal1} (ii). 

Under an additional assumption of asymptotic flatness on the state dynamics, i.e, $K < 0$ in 
(A1) (ii), the following is proven in [\cite{AriBorkarGhosh}, Chapter 7].
\begin{theorem}\label{BorkarErgodicresult2} Assume (A1) with $K <0$. Then there exists
	$\pi^* \in {\mathcal G}$ such that $\rho^* = \iint \bar{r}(x, u)  \pi^*(dx du)$ and if 
	$(X^*(\cdot), U^*(\cdot))$ denote a stationary solution pair corresponding to $\pi^*$, then
	given any $x \in \mathbb{R}^d$, 
	\[
	\liminf_{t \to \infty} \frac{1}{t}  E_x \Big[ \int^t_0 r(X(s), U^*(s)) ds \Big]
	\ = \ \rho^*,
	\]
	where $X(\cdot)$ denote the solution to (\ref{statedegensde}) corresponding to $U^*(\cdot)$ and
	initial condition $x \in \mathbb{R}^d$.
	More over 
	\[
	\rho(x, U^* (\cdot)) \leq \rho(x, U(\cdot)) , \ {\rm for\ all} \ x \in \mathbb{R}^d, 
	U(\cdot) \in {\mathcal U} .
	\]
\end{theorem}
Thus, Theorem \ref{BorkarErgodicresult2} gives the existence of optimal control in the 
sense of Definition \ref{ergodicoptimal1} (i). But neither the  theorem nor the method of
proof throw any light on the computation of the optimal control, though they give a characterization of the value $\rho^*$ as a viscosity solution of the corresponding HJB
equation in the asymptotic flat case, see Theorem 7.3.10 of \cite{AriBorkarGhosh}, p.260.

Our contribution is manyfold. Firstly we observe in Theorem \ref{metatheorem} that small noise 
limit  of the degenerate ergodic optimal control problem selects the ergodic optimal control problem which seeks minimization over all admissible pairs, 
i.e., Definition \ref{ergodicoptimal1} (ii). Also Theorem \ref{main2} and Theorem \ref{main3}
indicate that this need not be 
a solution for the ergodic control problem in Definition \ref{ergodicoptimal1} (i). i.e., small 
noise limit need not select in general the ergodic control problem in the sense of Definition
\ref{ergodicoptimal1} (i). This parallels the selection of physically relevant invariant 
probability measure from the set of all invariant measures/equilibrium  of a given  dynamical system. Hence we can observe that ergodic control problem in the sense of Definition
\ref{ergodicoptimal1} (ii) is the relevent problem when the underlying state dynamics seeks
'thermalization'. When the state dynamics of the ergodic control problem is a controlled 
deterministic gradient flow, 
we establish a tunneling behavior for the  one 
dimensional ergodic control problem and give a representation of the value $\rho^*$ in terms 
of the continuous time Markov chain which represents the tunneling of the controlled state dynamics corresponding to the ergodic optimal Markov control.

Secondly, we explore sufficient conditions for the small noise approximation of the 
ergodic control problem in the sense of Definition \ref{ergodicoptimal1} (i).
We show that  the small noise approximation of the ergodic control 
problem converges to ergodic optimal control problem in the sense of Definition \ref{ergodicoptimal1} (i) 
first in  the case when the state dynamics satisfies the asymptotic flaness condition given by (A1) with $K < 0$. Under a  different set of sufficient conditions, i.e. (B) and (D), we show the 
convergence to the ergodic optimal control problem in the sense of Definition \ref{ergodicoptimal1} (i) among the class of stationay Markov controls. 

Thirdly, we obtain error bounds for 
approximate controls which are ergodic optimal control 
for the perturbed non degenerate ergodic control problem whose optimal control can be 
characterized using the minimizing selector of corresponding HJB equation and hence computable,
see Theorem \ref{main1}.

The paper is organized as follows. In Section 2, we give a selection theorem for the 
degenerate ergodic control problems using small noise limit, see Theorem \ref{metatheorem}. 
In Theorem \ref{main2}, we show that
the optimal control  corresponding to the value $\rho^*$ is optimal in the sense of Definition \ref{ergodicoptimal1} (i) for the initial values $x $ in the support of the 
invariant probability measure of the sde (\ref{statedegensde}) correponding to optimal control. 
This in particular implies that $\rho^* = \rho^*_{\mathcal U}$.
In Section 3  we  address the problem when the state dynamics satisfies certain asymptotic flat condition. In Theorem \ref{asymptoticflatcase}, we show that small noise limit of the ergodic control problem is  the ergodic control problem in the sense of 
Definition \ref{ergodicoptimal1} (i) and also gives an error bound for the approximate optimal
controls. 
In Section 4, we consider the degenerate ergodic control problem under more general conditions.  
Under suitable conditions, in Theorem \ref{main1}, we show that the small noise limit 
of the ergodic control problem becomes Definition \ref{ergodicoptimal1} (i) among the class of stationay Markov controls.  We also 
show the existence of  $\varepsilon$-optimal controls over the space of stationary Markov
strategies in Theorem \ref{approximateoptimal}. Also under an additional condition, we obtain
error bounds in Theorem \ref{errorbound1}. Note that we were only showing the existence of
$\varepsilon$-optimal controls when minimization is over ${\mathcal U}_{SM}$. So we do not 
know whether the value  $\rho^*_{\mathcal  U}$ coinsides with the value 
$ \rho^*_{{\mathcal U}_{SM}}$. But we have affirmative answer for a special case in 
Theorem \ref{errorbound2}. In Section 5, we prove in Theorem \ref{tunnelvalue} which gives 
a characterization of the optimal ergodic value $\rho^*$ in terms of the underlying 
Markov chain which represents the tunneling of the optimaly controlled state dynamics. 
In Section 6, we give the proof of a characterization for invariant probability distribution
for solutions of sdes which are probably non Feller.

\section{Degenerate Ergodic control: General results} In this section, we give a selection
theorem for degenerate optimal control problems in the following sense. The small noise
perturbation limit of the ergodic control problem pick the ergodic value $\rho^*$ which is
the minimum among all ergodic optimal values given in (\ref{ergodicvalues}). 
Along with (A1), we assume the following
stability condition.

\begin{assumption*} [\bf L] The map 
	$(x, U) \mapsto {\mathcal L}^{U} \hat V (x)$ from $\mathbb{R}^d \times \mathbb{U}$ to $\mathbb{R}$ 
	is  inf-compact for some positive inf-compact $\hat V \in C^2(\mathbb{R}^d)$,
	where ${\mathcal L}^U$ is as in (\ref{generators}). 
\end{assumption*}

Consider the  small noise perturbation of the ergodic optimal control problem 
(\ref{statedegensde});(\ref{ergodiccost}) with state dynamics 
\begin{equation}\label{statedegenperturbsde}
	d X^\varepsilon (t) \  =  \  m(X^\varepsilon (t), U(t)) dt + \sigma_\varepsilon
	(X^\varepsilon (t)) d W(t),
\end{equation}
and cost criterion 
\begin{equation}\label{ergodiccost-perturbed}
	\rho_\varepsilon (x, U(\cdot)) \ = \  \liminf_{t \to \infty} \frac{1}{t} E_x \Big[ \int^t_0
	r(X^\varepsilon(s), U(s)) ds \Big] ,  \ U (\cdot) \in {\mathcal U}.
\end{equation}
where $X^\varepsilon(\cdot)$ is a  solution to (\ref{statedegenperturbsde})
satisfying $X^\varepsilon(0) =x$ corresponding to the
admissible control $U(\cdot)$. The choice of the perturbation $\sigma_\varepsilon$ of
$\sigma$ is as follows.  

\noindent For $\varepsilon > 0$, choose 
$\sigma_\varepsilon : \mathbb{R}^d \to  \mathbb{R}^d \otimes \mathbb{R}^d$ such that 
\[
(i) \, \sigma_\varepsilon \sigma^T_\varepsilon \geq \varepsilon I , \  (ii) 
\sup_{x \in \mathbb{R}^d} \| \sigma_\varepsilon (x)  - \sigma (x) 
	\| \leq \varepsilon |K|,  \ 
	(iii) \| \sigma_\varepsilon (x) - \sigma_\varepsilon (y) \| \leq \  |K| \| x - y\|, 
\]
 for some $K \in \mathbb{R}$. The constant $K $ is chosen without any loss of generality  as the constant in (A1)(ii). 
One such  choice is given by $\sigma_\varepsilon \sigma^T_\varepsilon = \sigma
\sigma^T + \varepsilon^2 I$, i.e. , $\sigma_\varepsilon$
is a small 'noise' perturbation of $\sigma$.

\begin{assumption*}[\bf H1]  For $\pi (dxdu) = u(x)(du) \eta (dx) \in {\mathcal G}$, 
	we assume that $\eta (dx)$ is a limit point of the invariant probability measure 
	$\eta^\varepsilon (dx)$ of the process $X^\varepsilon (\cdot)$ given by
	(\ref{statedegenperturbsde}) corresponding to $u(\cdot)$. 
\end{assumption*} 
\begin{assumption*}[\bf H2]   For each $u \in {\mathcal U}_{SM}$ and $\delta > 0$, there
	exist $u_\delta \in {\mathcal U}_{SM}$ which is continuous and a limit point 
	$\eta_\delta(dx)$  in ${\mathcal P}(\mathbb{R}^d)$ of the invariant probability measures 
	$\eta^\varepsilon_\delta(dx)$ of the process $X^\varepsilon_\delta(\cdot)$ given by
	(\ref{statedegenperturbsde}) corresponding to $u_\delta(\cdot)$ 
	such that 
	\[
	\iint \bar{r} (x, v) u_\delta(x)(dv) \eta_\delta(dx)  \ \leq \ \inf_{x \in \mathbb{R}^d} 
	\rho(x, u(\cdot)) + \delta .
	\]
	
\end{assumption*}

\begin{theorem}\label{metatheorem} Assume (A1),  (L) and (H1) or (H2). Then
	\[
	\liminf_{\varepsilon \to 0} \rho_\varepsilon (x, U^\varepsilon (\cdot)) = \rho^*, 
	x \in \mathbb{R}^d,
	\]
	where $U^\varepsilon(\cdot)$ is an ergodic optimal control for 
	(\ref{statedegenperturbsde});(\ref{ergodiccost-perturbed}).
\end{theorem} 
\begin{proof} Let $(X^*_\varepsilon(\cdot), U^*_\varepsilon (\cdot))$ be an optimal ergodic stationary pair and $\pi^\varepsilon \in \mathcal{P}(\mathbb{R}^d \times \mathbb{U})$ be the 
	corresponding ergodic ocupation measure for (\ref{statedegenperturbsde});
	(\ref{ergodiccost-perturbed}).
	Using (A1), it follows that the laws of $(X^*_\varepsilon(\cdot), U^*_\varepsilon (\cdot))$
	is tight and hence has a limit point  $(X(\cdot), U(\cdot))$.  Let the convergence be along the 
	subsequence $\{\varepsilon_n \}$. One can see that $(X(\cdot), 
	U(\cdot))$ is a stationary admissible pair for (\ref{statedegensde}). Let $\pi$ be its ergodic
	occupation measure, and hence $\pi^\varepsilon$ converges weakly to $\pi$ along the same
	subsequence $\{ \varepsilon_n\}$. Thus we have
	\[
	\lim_{n \to \infty} \rho_{\varepsilon_n} (x, U^*_{\varepsilon_n} (\cdot))
	\ = \lim_{n \to \infty} \iint \bar{r}(x, u) \pi^{\varepsilon_n} (dxdu) \ = \
	\iint \bar{r}(x, u) \pi (dxdu)
	\]
	Let ${\mathcal G}^*$ denote the set of all limit points of $\{\pi^\varepsilon \}$. Then
	from above, it follows that
	\begin{eqnarray}\label{eq1metathm}
		\liminf_{\varepsilon \to 0} \rho_\varepsilon (x, U^\varepsilon (\cdot)) 
		& = & \inf_{\pi \in {\mathcal G}^*} \iint \bar{r}(x, u) \pi(dxdu) := \rho^*_{inf}, \\ \nonumber
		\limsup_{\varepsilon \to 0} \rho_\varepsilon (x, U^\varepsilon (\cdot)) 
		& = & \sup_{\pi \in {\mathcal G}^*} \iint \bar{r}(x, u) \pi(dxdu) := \rho^*_{sup} .
	\end{eqnarray}
	Now assume (H1). 
	Let 
	$(X^*(\cdot), U^*(\cdot))$ be an optimal ergodic pair  as in Theorem \ref{BorkarErgodicresult1}
	and $u^*(\cdot)$ be a corresponding stationary Markov control. Then 
	$\pi^*(dx du) = u^*(x)(du) \eta^*(dx) \in {\mathcal G}$ for some invariant probability
	distribution $\eta^*(dx)$ of $X^*(\cdot)$. 
	Let $X^\varepsilon (\cdot)$ be the solution to (\ref{statedegenperturbsde})
	corresponding to $u^*(\cdot)$ with initial law $X^*(0)$. Using (L), a unique invariant
	probability  measure $\eta^\varepsilon (dx)$ of $X^\varepsilon (\cdot)$ exists. Then 
	$\pi^\varepsilon (dxdu) = u^*(x)(du) \eta^\varepsilon (dx)$ is an ergodic occupation measure 
	of $( X^\varepsilon (\cdot), u^*(\cdot))$. In view of  (H1), $\eta^*(dx)$ is a limit point
	of $\eta^\varepsilon (dx)$ and hence $\pi^*(dxdu)$ is a limit point of 
	$\pi^\varepsilon(dxdu)$. Let 
	$\pi^{\varepsilon_n} \to \pi^*$ in $\mathcal{P}(\mathbb{R}^d \times \mathbb{U})$
	as $n \to \infty$. 
	Since 
	\[
	\rho_\varepsilon (x, U^\varepsilon (\cdot)) \leq \iint \bar{r}(x, u) \pi^\varepsilon (dxdu),
	\]
	by letting $n \to \infty$, it follows that 
	\begin{eqnarray*}
		\liminf_{\varepsilon \to 0} \rho_\varepsilon (x, U^\varepsilon (\cdot)) 
				& \leq & \lim_{n \to \infty} \iint \bar{r}(x, u) \pi^{\varepsilon_n} (dxdu) 
		\ = \ \iint \bar{r}(x, u) \pi^* (dxdu) = \rho^{**} = \rho^* .
	\end{eqnarray*}
	The last equality above is due to Theorem \ref{BorkarErgodicresult1}. 
	Hence 
	\begin{equation}\label{eq2metathm}
		\rho^{**} \leq \rho^*_{inf}  \leq \rho^{**}.
	\end{equation}
	Hence using (\ref{eq1metathm}) and (\ref{eq2metathm}), we have 
	\[
	\liminf_{\varepsilon \to 0} \rho_\varepsilon (x, U^\varepsilon (\cdot)) = \rho^*.  
	\]
	
	Now assume (H2). 
	Let $u^*(\cdot) $   be as above. For $\delta > 0$, there exists 
	$u^*_\delta (\cdot)  \in {\mathcal U}_{SM}$, continuous and 
	$\eta_\delta(dx) \in {\mathcal P}(\mathbb{R}^d)$ such that 
	\[
	\iint \bar{r} (x, v) u^*_\delta(x)(dv) \eta_\delta(dx) \ \leq \ \rho(x, u^*_\delta (\cdot))
	+ \delta , x \in  \mathbb{R}^d .
	\]
	Let $X^\varepsilon_\delta (\cdot)$ be the solution to (\ref{statedegenperturbsde})
	corresponding to $u^*_\delta(\cdot)$ with initial law $X^*(0)$.
	Using (L), a uniqu invariant probability  measure $\eta^\varepsilon_\delta (dx)$ of 
	$X^\varepsilon_\delta (\cdot)$ exists. Then 
	$\pi^\varepsilon_\delta (dxdu) = u^*_\delta(x)(du) \eta^\varepsilon_\delta (dx)$ is an
	ergodic occupation measure of $( X^\varepsilon_\delta (\cdot), u^*_\delta(\cdot))$. 
	Consider
	\begin{eqnarray*}
		\liminf_{\varepsilon \to 0} \rho_\varepsilon (x, U^\varepsilon (\cdot)) 
		& \leq & \lim_{n \to \infty} \iint \bar{r}(x, u) \pi^{\varepsilon_n}_\delta (dxdu) 
		\ = \ \iint \bar{r}(x, u) u^*_\delta(x)(dv) \eta_\delta(dx)  
		\ \leq \ \rho^* + \delta.
	\end{eqnarray*}
	Since $\delta > 0$ is arbitrary, the proof is completed. 
\end{proof} 
\begin{remark} (i) In (H1), if $\eta(dx)$ is the limit of $\eta^\varepsilon (dx)$ or
	in (H2) $\eta_\delta (dx)$ is the limit of $\eta^\varepsilon_\delta (dx)$, then in 
	Theorem \ref{metatheorem}, $\liminf$ becomes $\lim$.
	
	(ii) We only need to assume (H2) for an optimal $u^* (\cdot)$ in Theorem \ref{metatheorem}.
	
\end{remark}

\ 

\begin{theorem}\label{main2} Assume (A1) and (L).
	Then there exists  $U^*(\cdot) \in {\mathcal U}$ and a stationary process $X^*(\cdot)$ 
	corresponding to $U^*(\cdot)$ satisfying  (\ref{statedegensde}) such that 
	\[
	\rho(x, U^*(\cdot)) \leq  \rho(x, U(\cdot)) \ {\rm for \ all} \ U(\cdot) \in {\mathcal U}, 
	x \in {\rm supp}(\eta^*),
	\]
	where $\eta^*$ is the law of $X^*(0)$. More over 
	\[
	\rho^* \ = \ \inf_{x \in \mathbb{R}^d} \rho^*_{\mathcal U} (x) .
	\]
\end{theorem}

\begin{proof} Using (A1), i.e.,  the map 
	$(x, U) \mapsto {\mathcal L}^{U} \hat V (x)$ from $\mathbb{R}^d \times \mathbb{U}$ is 
	inf-compact for some positive inf-compact $\hat V \in C^2(\mathbb{R}^d)$, 
	from Theorem 7.2.1, p.254, \cite{AriBorkarGhosh}, there exists an optimal 
	pair $(X^* (\cdot), U^*(\cdot)),$ with $U^*(t) = u^*(X^*(t)), t \geq 0$ and 
	$u^* \in {\mathcal U}_{SM}$. i.e. 
	\begin{equation}\label{optimalvalue}
		\rho^* = \inf_{(X(\cdot), U(\cdot))} \liminf_{ t \to \infty} \frac{1}{t} E \Big[ \int^t_0
		r(X(s), U(s)) ds \Big] = \liminf_{ t \to \infty} \frac{1}{t} E \Big[ \int^t_0
		r(X^*(s), U^*(s)) ds \Big],
	\end{equation}
	where infimum is over all admissible pairs of (\ref{statedegensde}).
	For the presribed control $U^*(\cdot)$, let $X(x;\cdot)$ denote the solution to (\ref{statedegensde})
	with initial condition $X(x;0) =x$. Then from the fact that conditional law of 
	$\{X^*(t) : 0 \leq t < \infty \}$ given $X^*(0) =x$ is same as the law of 
	$\{X(x;t) : 0 \leq t < \infty\}$, we have for each $t > 0$,
	\begin{eqnarray*}
		E \Big[\int^t_0 r(X^*(s), U^*(s)) ds \Big]   & = & \int_{\mathbb{R}^d}   E \Big[ \int^t_0
		r(X^*(s), U^*(s)) ds \Big| X^*(0) = x \Big] \, \eta^*(dx) \\
		& = & \int_{\mathbb{R}^d}   E \Big[ \int^t_0 r(X(x;s), U^*(s)) ds  \Big] \, \eta^*(dx),
	\end{eqnarray*}
	where $\eta^*(dx)$ denote the Law of $X^*(0)$. 
	Hence  using Fatou's lemma, we get,
	\begin{eqnarray*}
		\rho^* & = & \liminf_{t \to \infty}  \int_{\mathbb{R}^d}  \frac{1}{t}  
		E \Big[ \int^t_0 r(X(x;s), U^*(s)) ds  \Big] \, \eta^*(dx) \\
		& \geq &   \int_{\mathbb{R}^d} \Big( \liminf_{t \to \infty} \frac{1}{t}  
		E \Big[ \int^t_0 r(X(x;s), U^*(s)) ds  \Big] \Big) \, \eta^*(dx)
		\ = \ \int_{\mathbb{R}^d} \rho(x, U^*(\cdot)) \eta^* (dx) .
	\end{eqnarray*}
	Now from the definition of $\rho^*$, we have
	$
	\rho^* \leq \rho(x, U^*(\cdot)), \ {\rm for\ all} \ x \in \mathbb{R}^d.
	$
	Combining the above inequalities, we get
	$
	\rho^* = \rho(x, U^*(\cdot))  \  {\rm a.s.\ } \ ( \eta^* ) .
	$
	Thus we get,
	\begin{equation}\label{rho*identity}
		\rho^* = \inf_{x \in \mathbb{R}^d} \inf_{U(\cdot) \in {\mathcal U}} \rho(x, U(\cdot)).
	\end{equation}
\end{proof}

Observe that, we can  take $X(x; \cdot)$ in the above proof as a Markov process, 
since using Theorem 6.4.16, page 241 of 
\cite{AriBorkarGhosh}, there exists a Markov process in the marginal class of $X(x; \cdot)$. 
Let $\eta^x_t, t > 0 $ denote the family of emipirical measures of $X(x; \cdot)$, i.e.
\[
\int f(y) \eta^x_t (dy) \ = \ \frac{1}{t} E \Big[ \int^t_0 f(X(x; s)) ds, \ f \in 
C_b (\mathbb{R}^d).
\]
For $x \in \mathbb{R}^d$, let 
$\eta^x$ be a limit point of $\eta^x_t$. Set 
$
T_t f(x) = E f(X(x;t)), f \in C_b (\mathbb{R}^d).
$
Then one can see that $T_t f$ is Lipschitz continuous for all $f \in C^2_b(\mathbb{R}^d)$, see Lemma \ref{weakFeller}. Hence for $f \in C^2_b (\mathbb{R}^d)$, 
\begin{eqnarray*}
	\int T_t f(x) \eta^x (dx) & = & \lim_{n \to \infty} \int T_t f(x) \eta^x_{t_n} (dx)
	\ = \ \lim_{n \to \infty} \frac{1}{t_n} \int^{t_n}_0 T_{s +t} f(x) ds \\
	& = & \lim_{n \to \infty} \frac{1}{t_n} \int^{t_n +t }_0 T_s f(x) ds
	- \lim_{n \to \infty} \frac{1}{t_n} \int^{t }_0 T_s f(x) ds 
	\ = \ \int f(x) \eta^x(dx).
\end{eqnarray*}
Hence $\eta^x$ is an invariant distribution of $X(x ; \cdot)$. 
Since $X^*(\cdot)$ and $X(x; \cdot)$ has the same Markov semigroup, it follows 
from Theorem 5.1, \cite{Hairer2008} that either $\eta^x = \eta^*$ or $\eta^x$ and $\eta^*$ are singular to each other. 

Set
\begin{equation}\label{setD_0}
	D_0 \ = \ \bigcup_{\eta^* \in {\mathcal H}^*} 
	\Big\{ x \in \mathbb{R}^d \, \Big| \,  \int f(y) \eta^x (dy) = \int f(y)
	\eta^* (dy) \ {\rm for\ all } \ f \in C_b(\mathbb{R}^d) \Big\},
\end{equation}
where
\begin{eqnarray}\label{eta*}
	{\mathcal H}^* & = & \Big\{ \eta^* \in \mathcal{P}(\mathbb{R}^d) \,  \Big| \, 
	\eta^* \ {\rm satisfies}\  \eta^*(dx) = \pi^*(dx, \mathbb{ U}), \ {\rm for\ some} \\ \nonumber 
	&& \ \ \  \pi^* \in \  {\rm argmin}_{\pi \in {\mathcal G}} \iint \bar{r}(x, u) \pi(dx du) \Big\} .
\end{eqnarray}
\begin{theorem}\label{main3} Assume (A1), (L) and that 
	\[
	\bar{r} (x, u) \ = \ r_1 (x) + \bar{r}_2 (u), \ x \in \mathbb{R}^d, u \in \mathbb{U}.
	\]
	Then  $U^*(\cdot) \in {\mathcal U}$ given in Theorem \ref{main2} satisfies
	\[
	\rho(x, U^*(\cdot)) \leq  \rho(x, U(\cdot)) \ {\rm for \ all} \ U(\cdot) \in {\mathcal U}, 
	x \in D_0.
	\]
	More over, if $D_0 \subset \mathbb{R}^d$, then the above inequality need not hold for 
	$x \notin D_0$. 
\end{theorem}
\begin{proof} For $x \in D_0$, let $X(x; \cdot)$ denote the solution to (\ref{statedegensde})
	corresponding to $U^*(\cdot)$ and initial condition $x$. Then
\begin{eqnarray*}
\rho (x, U^*(\cdot)) & = & \liminf_{t \to \infty} \frac{1}{t} E \Big[ \int^t_0 r(X(x; s), U^*(s)) ds \Big] \\
& = & \liminf_{t \to \infty} \frac{1}{t} \Big\{ E \Big[ \int^t_0  r_1(X(x; s)) 
ds \Big] +  E \Big[ \int^t_0 \int_\mathbb{U}  \bar{r}_2(u)U^*(s)(du)  ds \Big] \Big\}\\
		& \leq &  \int_{\mathbb{R}^d} r_1 (x) \eta^* (dx) 
		+ \int_{\mathbb{R}^d} \int_\mathbb{U}  \bar{r}_2( u)\pi^*(dxdu) 
		\ = \ \rho^* .
	\end{eqnarray*}
	Since $\rho^* \leq \rho(x, U(\cdot)) $ for all $U (\cdot) \in {\mathcal U}$, it follows that
	\[
	\rho^* = \rho (x, U^*(\cdot)), x \in D_0.
	\]
	
	For $x \notin D_0$, for any $U^* (\cdot)$ which is an ergodic optimal stationary control, 
	the corresponding $\eta^x$ is singular to $\eta^* \in {\mathcal H}^*$. Hence when 
	$ r_1$  satisfies argmax $r_1 (x) \cap \bigcup_{\eta^* \in {\mathcal H}^*} 
	\ {\rm supp} (\eta^*) = \emptyset$, we get for $x \in D^c_0$ satisfying 
	\[
	\min_{ {\rm supp}(\eta_x) } r_1 \geq \max_{ {\rm supp}(\eta^*) } r_1,
	\]
	\begin{eqnarray*}
		\rho (x, U^*(\cdot)) & = & \liminf_{t \to \infty} \frac{1}{t} 
		E \Big[ \int^t_0 r(X(x; s), U^*(s)) ds \Big] \\
		& = &  \int_{\mathbb{R}^d} r_1 (x) \eta^x (dx) + \iint \bar{r}_2 (u) \pi^*(dxdu) 
		\ >  \ \rho^* ,
	\end{eqnarray*}
	where $\eta^x$ is the limit point corresponding to the liminf. 
	This completes the proof. 
\end{proof}

\section{Degenerate Ergodic control : Asymptotically Flat Diffusion} 
In this section,  we assume (A1) with $K <0$,
i.e. the controlled diffusions are asymptotically flat,
see Lemma  7.3.7, p.257 of  \cite{AriBorkarGhosh}.
We have the following result  from    \cite{AriBorkarGhosh}, Theorem 7.3.7,
p.257, Theorem 7.3.9, p.259, Theorem 7.3.10, p.260.
\begin{theorem}\label{existenceergodicoptimal-perturbed} Assume (A1) with $K < 0$.
	Then there exists $ U^* (\cdot)  \in {\mathcal U}_{SM}$ such that 
	\begin{eqnarray*}
		\rho^*   & =  & \rho^*_{\mathcal U} (x) = \  \rho (x, U^* (\cdot)),  x \in \mathbb{R}^d . 
	\end{eqnarray*}
	Also there exists $\varphi \in C^{0,1 }(\mathbb{R}^d)$ such that 
	$(\rho^*, \varphi) \in \mathbb{R} \times C^{0, 1} (\mathbb{R}^d) $  is a 
	viscosity solution to the HJB equation 
	\[
	\rho \ =  \  \inf_{ U \in \mathcal{P}(\mathbb{U})}[  \langle m(x, U), \  \nabla \varphi \rangle 
	+ r(x, U) ] + \frac{1}{2} {\rm trace}(a(x) \nabla^2 \varphi) , x \in \mathbb{R}^d.
	\]
	More over, if  $(\rho, \psi) \in \mathbb{R} \times C^{0, 1} (\mathbb{R}^d) $  is a
	viscosity solution to the HJB equation 
	\[
	\rho \ =  \  \inf_{ U \in \mathcal{P}(\mathbb{U})}[  \langle m(x, U), \  \nabla \psi \rangle 
	+ r(x, U) ] + \frac{1}{2} {\rm trace}(a(x) \nabla^2 \varphi) , x \in \mathbb{R}^d,
	\]
	then $\rho = \rho^*$. 
\end{theorem}
i.e., as opposed to the non degenerate case, it is not known whether  any minimizing selector 
is an optimal stationary  Markov control.   
Also Theorem \ref{existenceergodicoptimal-perturbed}, only guarantee uniqueness of the value
$\rho^*$ but not $\varphi(\cdot)$. Using small noise perturbation to a specific example,
first we illustrate  what more one can expect.

\begin{lemma}\label{momentestimate2} Assume (A1). For a prescribed control 
	$U(\cdot) \in {\mathcal U}$, 
	let $X^\varepsilon (\cdot), X(\cdot)$ denote respectively, solutions 
	to (\ref{statedegenperturbsde}), (\ref{statedegensde}) with initial condition $x \in \mathbb{R}^d$. Then for $t \geq 0$ ,
	\[
	E \| X^\varepsilon (t) - X(t) \|^2 \ \leq \ \left\{
	\begin{array}{lll}
		\frac{\varepsilon^2}{2 K} ( e^{2 K t} - 1) & {\rm if} & K \neq 0 \\
		\varepsilon^2 t & {\rm if} & K =0 .
	\end{array}
	\right.
	\]
	
\end{lemma}

\begin{proof} We have
	\[
	X^\varepsilon (t) - X(t) \  = \ \int^t_0 (m(X^\varepsilon (s), U(s)) - m(X(s), U(s)) ) ds 
	+ \int^t_0 (\sigma_\varepsilon (X^\varepsilon(s)) - \sigma (X(s)))  d W(s) .
	\]
	Using It$\hat{{\rm o}}$'s formula we get
	\begin{eqnarray}\label{eq1gencase}
		d \|X^\varepsilon (t) - X(t) \|^2 & = & 2 \sum^d_{i=1}  (m_i(X^\varepsilon
		(t), U(t)) - m_i(X(t), U(t) )) (X^\varepsilon_i (t) - X_i(t)) dt  \nonumber \\
		&& + \sum^d_{i,j=1} (\sigma^{ij}_\varepsilon (X^\varepsilon (t) -
		\sigma^{ij} (X(t))^2 d t \\ \nonumber 
		&& +  2 \sum^d_{i,j=1} (X^\varepsilon_i (t) - X_i(t))(
		\sigma^{ij}_\varepsilon (X^\varepsilon (t) - \sigma^{ij} (X(t))) d W_j (t) .
	\end{eqnarray} 
	Using (A1) (ii), we get 
	\begin{eqnarray}\label{eq2gen}
		E\|X^\varepsilon (t) - X(t) \|^2 & \leq & E\|X^\varepsilon (s) - X(s) \|^2 + 
		2 K \int^t_s E \|X^\varepsilon (s') - X(s')\|^2 ds' \nonumber \\
		&&  +  \varepsilon^2 (t -s), \ 0 \leq s < t < \infty . 
	\end{eqnarray}
	
	Now using  standard comparison theorem for odes, we get
	\[
	E \|X^\varepsilon (t) - X(t) \|^2  \leq h(t), \ 0 \leq t < \infty ,
	\]
	where $h(\cdot)$ is the solution to the ode
	\[
	\dot{h}(t) = 2 K h(t) + \varepsilon^2 , \  h(0) = 0.
	\]
	Hence for $t \geq 0$, 
	\[
	E \|X^\varepsilon (t) - X(t) \|^2  \leq \left\{
	\begin{array}{lll}
		\frac{\varepsilon^2}{2 K} (e^{2K t} -1) & {\rm if} & K \neq 0 \\
		\varepsilon^2 \, t & {\rm if} & K =0 .
	\end{array}
	\right.
	\]
	
\end{proof}

\begin{theorem} \label{asymptoticflatcase} Assume (A1) with $K < 0$. 
	
	\noindent (i) Following
	inequality holds. 
	\[
	- \hat K \,  \varepsilon  \leq \rho_\varepsilon(x, U^*_\varepsilon (\cdot)  - \rho(x, U^*(\cdot))
	\leq \hat K \, \varepsilon, \ x \in \mathbb{R}^d, 
	\]
	for some constant  $\hat K$ which depends only on $K$ and the Lipschitz constant of
	$r$, where $U^*_\varepsilon (\cdot), U^*(\cdot)$ are  optimal control for  (\ref{statedegenperturbsde});(\ref{ergodiccost-perturbed}) and 
	(\ref{statedegensde});(\ref{ergodiccost}) respectively. 
	
	\noindent  (ii)  For $x \in \mathbb{R}^d$, 
	\[
	\lim_{\varepsilon \to 0} \rho(x, U^*_\varepsilon (\cdot)) \  = \ \rho(x, U^*(\cdot)).
	\]
\end{theorem}
\begin{proof} 
	Consider
	\begin{eqnarray*}
		\rho^*_\varepsilon (x)  - \rho^*_{\mathcal U} (x) & \geq & \lim_{t \to \infty} \frac{1}{t} E_x \Big[
		\int^t_0 r(X^\varepsilon (s), U^*_\varepsilon (s)) ds \Big] \, -  \,
		\liminf_{t \to \infty}  \frac{1}{t} E_x \Big[ \int^t_0 r(X (s), U^*_\varepsilon (s))
		ds \Big] \\
		& \geq & \liminf_{t \to \infty} \frac{1}{t} E \Big[ \int^t_0 ( r(X^\varepsilon (s),
		U^*_\varepsilon (s)) -  r(X (s), U^*_\varepsilon (s))) ds \Big] \\
		& \geq & {\rm Lip}(r) \, \liminf_{t \to \infty} \Big( - \frac{1}{t}  \int^t_0E 
		\|X^\varepsilon (s) - X (s)\|  ds\Big)  \\
		& \geq & - {\rm Lip}(r) \,  \sqrt{ \frac{\varepsilon}{2 |K|} },
	\end{eqnarray*}
	where $X^\varepsilon (\cdot), X(\cdot)$ denote the solutions of the sdes 
	(\ref{statedegenperturbsde}), (\ref{statedegensde}) with initial condition $x \in \mathbb{R}^d$
	corresponding to the control $U^*_\varepsilon (\cdot)$ in the prescribed form.  The last inequality
	above follows from Lemma \ref{momentestimate2}. Again consider 
	\begin{eqnarray*}
		\rho^*_\varepsilon (x)  - \rho^*_{\mathcal U} (x) & \leq & 
		\liminf_{t \to \infty} \frac{1}{t} E_x\Big[ \int^t_0 r(X^\varepsilon (s), U^*(s)) ds \Big] 
		\, -  \, \liminf_{t \to \infty}  \frac{1}{t} E_x \Big[ \int^t_0 r(X (s), U^*(s)) ds \Big] \\
		& \leq & \limsup_{t \to \infty} \frac{1}{t} E \Big[ \int^t_0 ( r(X^\varepsilon (s),
		U^*_\varepsilon (s)) -  r(X (s), U^*_\varepsilon (s))) ds \Big] \\
		& \leq & {\rm Lip}(r) \, \limsup_{t \to \infty}  \frac{1}{t}  \int^t_0E 
		\|X^\varepsilon (s) - X (s)\|  ds \\
		& \leq &  {\rm Lip}(r) \, \sqrt{ \frac{\varepsilon}{2 |K|} } ,
	\end{eqnarray*}
	where $X^\varepsilon (\cdot), X(\cdot)$ denote the solutions of the sdes 
	(\ref{statedegenperturbsde}), (\ref{statedegensde}) with initial condition $x \in \mathbb{R}^d$
	corresponding to the control $U^* (\cdot)$ in the prescribed form and the last inequality
	follows from Lemma \ref{momentestimate2}. 
	This completes the proof of inequality in (i). 
	
	>From the above calculation, we can see that
	\[
	\rho_\varepsilon (x, U^*_\varepsilon (\cdot)) - \rho (x, U^*_\varepsilon (\cdot)) 
	\geq - {\rm Lip}(r) \,  \sqrt{ \frac{\varepsilon}{2 |K|} }.
	\]
	Hence we have from (i), 
	\[
	\limsup_{\varepsilon \to 0} \rho(x, U^*_\varepsilon (\cdot)) \leq \rho(x, U^*(\cdot)) , 
	x \in \mathbb{R}^d.
	\]
	Since $\rho(x, U^*_\varepsilon (\cdot)) \geq \rho(x, U^*(\cdot)) $ for all $x \in \mathbb{R}^d$, 
	(ii) follows. 
	
\end{proof}

\section{Beyond Asymptotically flat diffusion} In this section, we go beyond the asymptotically flat diffusions framework. We assume that

\begin{assumption*}[\textbf{B}]
	\begin{itemize}
		\item There exists  uniformly continuous $b_i = (b^1_i, \cdots, b^d_i) :
		\mathbb{R}^d \to \mathbb{R}^d$ satisfying linear growth condition such that 
		\[
		b^i_1 (x-y) < \bar m_i (x, u) - \bar m_i(y, u) \leq b^i_2 (x-y), \  x, y \in \mathbb{R}^d, u
		\in \mathbb{U},
		\]
		where $b_i, i =1,2 $ is such that  
		\[
		\dot{X}^i(t) = b_i (X^i(t)) , i =1,2
		\]
		has  a unique stable equilibrium at $0$ and $\mathbb{R}^d$ is  attracted to $0$. 
		
		\item There exist non-negative and inf-compact $h \in C(\mathbb{R}^d)$ and a constant $k_0$ satisfying

		\begin{equation}\label{stabilityassump}\mathcal{L}_{\epsilon, i}V(x) \le k_0 - h(x), \; \forall x \in \mathbb{R}^d, \end{equation} where $V$ is non-negative, inf-compact and $V \in C^2(\mathbb{R}^d),$ such that $\nabla^2 V$ is bounded,  
		\begin{equation}\label{generatorauxiliary}
			\mathcal{L}_{\epsilon, i} \, \phi := \langle b_i(x), \nabla \phi \rangle + \frac{\varepsilon^2}{2} \mathrm{tr}(\hat \sigma (0) \hat \sigma (0)^T \nabla^2 \phi), 
		\end{equation}
		$\hat \sigma_\varepsilon : \mathbb{R}^d \times \mathbb{R}^d \to \mathbb{R}^d \otimes \mathbb{R}^d$
		defines the perturbation  $\sigma_\varepsilon$ of $\sigma$ as given by 
		\[
		\hat \sigma_\varepsilon (x, y)  = \sigma (y)  + \varepsilon \hat \sigma (x-y)  , \, x, y \in \mathbb{R}^d , \ \sigma_\varepsilon (x) = \hat \sigma_\varepsilon (x, x) 
		\]
		and $\hat \sigma: \mathbb{R}^d \to \mathbb{R}^d \otimes \mathbb{R}^d  $ is uniformly elliptic
		and Lipschitz continuous.
	\end{itemize}
\end{assumption*}

\begin{assumption*}[\textbf{C}] The function  $r$ is bounded, and $b_i : \mathbb{R}^d \to \mathbb{R}^d, \ \hat \sigma : \mathbb{R}^d \to \mathbb{R}^d \otimes \mathbb{R}^d$ 
	are bounded, smooth with bounded derivetives of order upto $2$. Also satisfies
	
	\begin{itemize}
		\item 
		\[
		\lambda \|x\|^2 \leq x^T \hat \sigma (0) \hat \sigma (0)^T x \leq \Lambda \|x\|^2, x \in \mathbb{R}^d
		\]
		for some $0 < \lambda \leq \Lambda$. 
		
		\item  There exists constants $\alpha > 0, \ 0 < \beta \leq 1$ such that
		\[
		\limsup_{\|x\| \to \infty} \Big[ \alpha \Lambda + \frac{1}{\|x\|^\beta} 
		\langle b_i (x) , x \rangle \Big] < 0 .
		\]
	\end{itemize}
\end{assumption*}

We consider the small noise limit of the  control problem  with state dynamics given by
\begin{equation}\label{gen_perturb_sde}
	d X^\varepsilon (t) = m(X^\varepsilon (t), U(t)) dt + \hat \sigma_\varepsilon 
	(X^\varepsilon (t), X^\varepsilon(t)) d W(t)
\end{equation}
and cost criterion given by (\ref{ergodiccost-perturbed}).

We will be using the following comparison theorem for multi dimensional sdes. The proof 
essentially follows from the proof of Theorem 1.1 of   \cite{ChristelRalf}. The key idea in the 
proof of Theorem 1.1 is application of It$\hat {\rm o}$'s formula to the difference of the processes
under consideration which turns out be same as ours. So we omit the details.
In the theorem below, we use $x \leq y, x, y \in \mathbb{R}^d$ for $x_i \leq y_i$ for all $i$ where $x_i, y_i$ respectively denote the $i$th component. 
\begin{theorem}\label{comparisonsde1} Let $X(\cdot), Y(\cdot), Z(\cdot)$ be 
	$\{ {\mathcal F}_t\}$-adapted processes with a.s. continuous paths on a complete probability
	space $(\Omega, {\mathcal F}, P)$
	such that $\{{\mathcal F}_t\} $ satisfies the usual conditions. Further, 
	let $X(\cdot), Y(\cdot)$ be pathwise solutions
	of the sdes
	\begin{eqnarray*}
		dX(t) & = & a(t, X(t)) dt + (Z (t) + \sigma (X(t)) ) dW(t), \\
		dY(t) & = & b(t, Y(t)) dt + (Z (t) + \sigma (Y(t)) ) dW(t),
	\end{eqnarray*}
	where $a= (a_1, \dots , a_d) , b = (b_1, \dots, b_d)  : [0, \infty) \times \mathbb{R}^d \to
	\mathbb{R}^d, $ are continuous functions, $\sigma : \mathbb{R}^d \to \mathbb{R}^d \otimes
	\mathbb{R}^d$ is Lipschitz continuous and $W(\cdot)$ is a $\{{\mathcal F}_t\}$-Wiener process 
	in  $\mathbb{R}^d$. If 
	\[
	a_i (t, x) < b_i(t, y) \ {\rm whenever} \ x_i = y_i , \ x_j \leq y_j , j \neq i
	\]
	and $X(0) \leq Y(0)$ a.s., then $P ( X(t) \leq Y(t), t \geq 0 ) =1$.
\end{theorem}

Consider for $U(\cdot) \in {\mathcal U}$, 
\begin{equation} \label{auxiliarysdeY}
	dY_i^\epsilon (t) =b_i(Y_i^\epsilon(t))dt+ ( \sigma(X^\varepsilon (t) - \sigma(X(t)) + 
	\epsilon \hat \sigma (0) ) dW(t),  \; i=1, 2 , 
\end{equation}
where  the processes $X^\varepsilon(\cdot)$ and $X(\cdot)$ are given by the solutions of the sdes (\ref{gen_perturb_sde}) and (\ref{statedegensde}) corresponding to $U(\cdot)$.
Following observation is kept in mind for the mext Lemma. Given $U(\cdot) \in {\mathcal U}$, the laws of $X^\varepsilon (\cdot)$ is tight and any limit point $\hat X(\cdot)$ in law is a
solution to the sde (\ref{statedegensde}) corresponding to $U(\cdot)$ when 
$U(\cdot)$ is either a prescribed control or a continuous Markov control. When $U(\cdot)$ is a continuous Markov control, then there may be multiple limit points which are all weak  solutions of (\ref{statedegensde}).

Set for $z = (y, x_1, x_2) \in \mathbb{R}^{3d}, u \in {\mathcal U}_{SM}$, 
\begin{equation}\label{generatorZeps}
	{\mathcal L}^\varepsilon_Z f (z) = \mathcal{L}_{\varepsilon, i} f + 
	\langle m(x_1, u( x_2)), \nabla_{x_1} f \rangle + \langle m(x_2, u(x_2)), 
	\nabla_{x_2} f \rangle + {\mathcal L}^1_{Z, \varepsilon} f(z),
\end{equation}
where 
\begin{eqnarray*}
	{\mathcal L}^1_{Z, \varepsilon} f & = & \frac{1}{2} {\rm trace}(\sigma(x_1, x_2) \sigma(x_1, x_2)^T \nabla^2_y f)
	+ \varepsilon \,  {\rm trace}(\sigma(x_1, x_2) \hat \sigma(0)^T \nabla^2_y f) \\
	&& + \frac{1}{2}  {\rm trace}(\sigma(x_1) \sigma(x_1)^T \nabla^2_{x_1} f) 
	+ \frac{1}{2} {\rm trace}(\sigma( x_2) \sigma(x_2)^T \nabla^2_{x_2} f) \\
	&& +   {\rm trace}(\sigma(x_1)  \sigma(x_2)^T \nabla^2_{x_1 x_2} f) 
	+  {\rm trace}(\sigma(x_1, x_2) \sigma(x_1)^T \nabla^2_{x_1 y} f)\\
	&& + \varepsilon \,  {\rm trace}(\hat \sigma(0 )  \sigma(x_1)^T \nabla^2_{x_1 y} f) 
	+  {\rm trace}(\sigma(x_1, x_2) \sigma(x_2)^T \nabla^2_{x_2 y} f) \\
	&& + \varepsilon \,  {\rm trace}(\hat \sigma(0 )  \sigma(x_2)^T \nabla^2_{x_2 y} f), \ 
	f \in C^2_b(\mathbb{R}^{3d}), 
\end{eqnarray*}
and $\mathcal{L}_{\varepsilon, i} $ is given by (\ref{generatorauxiliary}). 

\begin{lemma}\label{asymtoticbehavior} Assume (A1) and  (B). 
	For $U (\cdot) \in {\mathcal U}$, given by $U(t) = u(  X(t)), t \geq 0,$ where $u$ is a measurable map such that 
	$(Y^\varepsilon_i(\cdot), X^\varepsilon (\cdot), X(\cdot))$ be solutions of (\ref{auxiliarysdeY}), 
	(\ref{gen_perturb_sde}) and (\ref{statedegensde}) respectively satisfying $X^\varepsilon (\cdot) \to X(\cdot)$ in law along a subsequence as $\varepsilon_n \to 0$. Then
	any limit point $\eta_{\varepsilon_n , i}$ of the empirical measures of 
	$Y^{\varepsilon_n}_i (\cdot)$ converges weakly to $\delta_0$.
	
	
\end{lemma}
\begin{proof} Fix $i=1$ and suppress $n$ from $\varepsilon_n$ and we denote $\mathcal{L}_{\epsilon, 1}$ by $\mathcal{L}_{\epsilon}$.\\
	Set $Z^\varepsilon (t) = (Y^\varepsilon_1 (t), X^\varepsilon (t), X(t))^T, t \geq 0$ and 
	$\sigma (x_1, x_2) = \sigma(x_1) - \sigma(x_2), x_1, x_2 \in \mathbb{R}^d$. 
	Then  the infinitesimal generator of $Z^\varepsilon (\cdot)$ is given by the differential operator 
	${\mathcal L}^\varepsilon_Z $ defined  in (\ref{generatorZeps}).

	Consider the family of empirical measures $\mu^\varepsilon_t \in \mathcal{P}(\mathbb{R}^{3d})$ of
	$Z^\varepsilon (\cdot)$, i.e., 
	\[
	\iiint_{\mathbb{R}^{3d}} f(z) \mu^\varepsilon_t (dz) \ = \ \frac{1}{t} E \Big[ 
	\int^t_0 f(Z^\varepsilon_1 (s)) ds \Big], \ f \in C_b(\mathbb{R}^{3d}), \, t > 0 .
	\]
	Let $\tau_n$ denote the exit time of the process $Z^\varepsilon (\cdot)$ from  $B_n$, the ball 
	in $\mathbb{R}^{3d}$ of radius $n$ centered at the origin. Define
	\[
	\hat V (z) = V(y) + V(x_1) + V(x_2), z = (y, x_1, x_2) \in \mathbb{R}^{3d},
	\]
	where $V$ is the Lyapunov function given in (B). Then observe that 
	\[
	{\mathcal L}^\varepsilon_Z \hat V (z) \ = \ \mathcal{L}_{\epsilon} V (y) 
	+ \mathcal{L}^1_{\epsilon} V(x_1) + \mathcal{L}^2_{\epsilon} V(x_2),
	\]
	where
	\[
	\mathcal{L}^i_{\epsilon} f  \ = \ \langle m(x_i, u(x_2)), \nabla f \rangle 
	+  \frac{1}{2}  {\rm trace}(\sigma(x_i) \sigma(x_i)^T \nabla^2 f), \ f \in C^2 (\mathbb{R}^d), 
	\ i =1,2.
	\]
	Then using It$\hat {\rm o}$-Dynkin's formula, we get
	\begin{eqnarray*}
		E_z [ \hat V((Z^\varepsilon(\tau_n \wedge t))] & = & \hat V(z) + E_z \Big[ \int^{\tau_n \wedge t}_0 
		{\mathcal L}^{\varepsilon}_Z \hat V(Z^\varepsilon (s)) ds \Big] \\
		& \leq & V(z) + K_0 E_z [\tau_n \wedge t]  - E_z \Big[ \int^{\tau_n \wedge t}_0 
		h(Y^\varepsilon_1 (s)) ds \Big]\\
		&&- E_z \Big[ \int^{\tau_n \wedge t}_0  h(X^\varepsilon (s)) ds \Big]
		- E_z \Big[ \int^{\tau_n \wedge t}_0  h(X (s)) ds \Big], 
	\end{eqnarray*} 
	where $Z^\varepsilon (0) =z$, the constant $K_0 > 0$ depends only on $k_0$ and the bounds of 
	$\sigma, \hat \sigma, $ and $\nabla^2 V$.
	Now by letting $n \to \infty$ with the help of Monotone convergence theorem, we get for each $R > 0$,
	\[
	E_z \Big[ \int^t_0 (h (Y^\varepsilon_1 (s))+ h(X^\varepsilon (s)+
	h(X(s))  I_{ \|Z^\varepsilon (s)\| \geq R } ds \Big]
	\leq \ V(z) + K_0 t .
	\]
	This implies
	\begin{equation}\label{asymtoticbehavioreq1}
		\frac{1}{t} \int^t_0 P\Big( \| Z^\varepsilon (s) \| \geq R \Big) ds 
		\ \leq \ \frac{K_0 t + V(z)}{3t \,  \inf_{\|x \| \geq R} h(x) }.
	\end{equation}
	Since $\inf_{\|x \| \geq R} h(x) \to \infty$ as $R \to \infty$, from (\ref{asymtoticbehavioreq1}),
	the tightness of the empirical measures $\{ \mu^\varepsilon_t : t \geq 0 \}$ follows for 
	each $\varepsilon > 0$. 
	
	Let $\mu^\varepsilon \in \mathcal{P}(\mathbb{R}^{3d})$ be a limit point of 
	$\{ \mu^\varepsilon_t : t \geq 0 \}$ for each $\varepsilon > 0$. Then it follows that 
	\begin{equation}\label{asymtoticbehavioreq2}
		\iiint {\mathcal L}^{\varepsilon}_Z f (z) \mu^\varepsilon (dz) = 0 \ {\rm for\ all} \ f \in C^\infty_c (\mathbb{R}^{3 d}).
	\end{equation}
	Hence,  $\mu^\varepsilon$ is an invariant probability measure for $Z^\varepsilon (\cdot)$, see Lemma \ref{invariantmeasurechar1}.
	
	Again  from (\ref{asymtoticbehavioreq1}), it follows that $\{ \mu^\varepsilon | \varepsilon > 0 \}$ is
	a tight family and let $\mu \in \mathcal{P}(\mathbb{R}^{3d})$ be a limit point of it.
	For each  $f \in C^\infty_c (\mathbb{R}^{3d})$, it is easy to see that 
	${\mathcal L}^{\varepsilon}_Z f \to {\mathcal L}^{0}_Z f $ uniformly, 
	where ${\mathcal L}^{0}_Z $  is the differential operator given by substituting $\varepsilon =0$ in 
	${\mathcal L}^{\varepsilon}_Z$. Note that for $f \in C^2_b (\mathbb{R}^{2d})$, we have
	\begin{eqnarray}\label{generatorZ}
		{\mathcal L}^0_Z f (x_1, x_2) & = &  \langle m(x_1 , u( x_2), \nabla_{x_1} f \rangle + 
		\frac{1}{2}  {\rm trace}(\sigma(x_1) \sigma(x_1)^T \nabla^2_{x_1} f) \nonumber \\
		&& + \langle m(x_2 , u( x_2), \nabla_{x_2} f \rangle 
		+ \frac{1}{2}  {\rm trace}(\sigma(x_1) \sigma(x_1)^T \nabla^2_{x_1} f) \\ \nonumber
		&& + {\rm trace}(\sigma(x_1)  \sigma(x_2)^T \nabla^2_{x_1 x_2} f)
	\end{eqnarray}
	which defines the infinitesimal generator of the process $(X(\cdot), \hat X(\cdot))$ given by
	\begin{eqnarray}
		d X(t) & = & m(X(t), u(X(t))) dt + \sigma (X(t)) d W(t), \label{Xsde1}\\
		d \hat X(t) & = & m(\hat X(t), u(X(t))) dt + \sigma (\hat X(t)) d W(t) \label{Xsde2}.
	\end{eqnarray}
	Since sde (\ref{Xsde2}) has a unique solution given by $X(\cdot)$ for any given solution 
	$X(\cdot)$ of 
	(\ref{Xsde1}) corresponding to any given initial condition $x \in \mathbb{R}^d$, it follows that
	${\mathcal L}^0_Z $ given in (\ref{generatorZ}) defines the infinitesimal generator for 
	$(X(\cdot), X(\cdot))$.
	
	For $f \in C^\infty_c (\mathbb{R}^{3d})$, by letting $\varepsilon \to 0$ in 
	(\ref{asymtoticbehavioreq2}), it follows that
	\begin{equation}\label{asymtoticbehavioreq3}
		\iiint  {\mathcal L}^{0}_Z f (z) \mu (dz) = 0 ,
	\end{equation}
	By disintegrating $\mu^\varepsilon, \mu$, we have
	\[
	\mu^\varepsilon (dy dx_1 dx_2) = g^\varepsilon (dy | x_1 , x_2) \nu^\varepsilon (dx_1 dx_2), \
	\mu (dy dx_1 dx_2) = g (dy | x_1 , x_2) \nu (dx_1 dx_2),
	\]
	where $\nu^\varepsilon$ is an invariant probability measures of
	$(X^\varepsilon (\cdot), X(\cdot)) $. Also from (\ref{asymtoticbehavioreq3}), it follows that
	\[
	\iint {\mathcal L}^0_Z f(x_1, x_2) \nu(dx_1 dx_2) = 0 , \ {\rm for\ all}\ 
	f \in C^\infty_c (\mathbb{R}^{2d}).
	\]
	Hence $\nu$ is an invariant measure for the process $(X(\cdot), X(\cdot))$. 
	Therefore,  $\nu $ is 
	supported in $\{ (x_1, x_2) | x_1 = x_2 \}$. Hence (\ref{asymtoticbehavioreq3}) takes the form
	\begin{equation}\label{asymtoticbehavioreq4}
		\int \langle b_1(y) , \, \nabla f (y)\rangle \, \eta (dy)  \ = \ 0 \  {\rm for\ all} \ 
		f \in C^\infty_c(\mathbb{R}^d).
	\end{equation}
	where $\eta(dy)$ comes from the distintegration of $\mu$ given by
	\[
	\mu (dy dx_1 dx_2) = h(dx_1 dx_2 | y) \eta (dy)
	\]
	and clearly $\eta_{\varepsilon, 1} \to \eta$ as $\varepsilon \to 0$. 
	>From (\ref{asymtoticbehavioreq4}), we get that $\eta (dy)$ is an invariant probability measure for 
	\[
	d Y(t) = b_1 (Y(t)) dt.
	\]
	and hence $\eta (dy) = \delta_0 (dy)$, using assumption (B). In particular, the conclusion of Lemma follows.
\end{proof}
\begin{remark}\label{lemma_asymtoticbehavior}\begin{itemize}
		\item In Lemma \ref{asymtoticbehavior}, when $u$ is continuous, then if $X(\cdot)$ is limit 
		point in law of $X^\varepsilon (\cdot)$, then one can see that $X(\cdot)$ solves the sde (\ref{statedegensde}) corresponding to the stationary Markov control $u(\cdot)$. But this is not 
		known to be true when $u(\cdot)$ is just measurable. 
		
		If $u(\cdot)$ is Lipschitz continuous, then the sde (\ref{statedegensde}) has a unique solution
		and hence $X^\varepsilon (\cdot)$ converges in law to the solution of the sde (\ref{statedegensde}).
		
		\item Even when $u$ is smooth, the empirical measures of $Y^\varepsilon_i (\cdot)$ can have multiple limit points but Lemma \ref{asymtoticbehavior} guarantees that all these limit points
		converges to $\delta_0$ as $\varepsilon \to 0$. 
	\end{itemize}
\end{remark}

For $R > 0$, define $h_R(s)=
\begin{cases}
	{\rm Lip}(r)s,  & |s| < R \\
	{\rm Lip}(r) R, & |s| \geq R.
\end{cases}$

\noindent For $u \in {\mathcal U}_{SM}$, if $X(\cdot)$ is a solution to the sde (\ref{statedegensde}) 
with initial condition $X(0) = x$, then 
\[
\rho(x, U(\cdot) ) = \liminf_{t \to \infty} \frac{1}{t} E_x \Big[ \int^t_0 r(X(s), U(s)) ds \Big],
\]
where $U(t) = u(X(t)), t \geq 0$ is in prescribed form and note that $X(\cdot)$ is
a unique solution to the sde (\ref{statedegensde}) corresponding to the prescribed control 
$U(\cdot)$.  Now let $X^\varepsilon (\cdot)$ denote a unique solution to the sde 
(\ref{gen_perturb_sde}) corresponding to $U(\cdot)$ with initial condition 
$X^\varepsilon (0) =x$,
then $X^\varepsilon (\cdot) $ converges in law to $X(\cdot)$ as $\varepsilon \to 0$. 
Hence,  
\begin{eqnarray} \label{eqlimit01}
	\frac{1}{t} E \Bigg[\int_0^t (r(X^\varepsilon(s),U(s))-r(X(s), U(s)))ds\Bigg] & 
	\nonumber\\ \nonumber 
	\ = \frac{1}{t} E \Bigg[\int_0^t (r(X^\varepsilon(s),U(s))-r(X(s), U(s)))
	I_{\|X^\varepsilon(s)-X(s)\| < R} ds\Bigg] &  \\ \nonumber 
	\ + \frac{1}{t} E \Bigg[\int_0^t (r(X^\varepsilon(s),U(s))-r(X(s), U(s))) 
	I_{\|X^\varepsilon(s) - X(s)\| \geq R} ds\Bigg] & \\ \nonumber 
	\ \leq  \frac{1}{t} E \Bigg[\int_0^t h_R(\|X^\varepsilon(s) - X(s)\|)ds\Bigg] & \\ \nonumber 
	\ + 2\|r\|_\infty \frac{1}{t} \int_0^t P(\|X^\varepsilon(s) - X(s)\| \geq R)ds   & \\
	\ =: I_1(R) + I_2(R) & 
\end{eqnarray}
Set $Y^\varepsilon (t) = X^\varepsilon (t) - X(t), t \geq 0$. Then $Y^\varepsilon (\cdot)$ is a
solution to the sde 
\begin{eqnarray} \label{sdeY}
	d Y^\varepsilon (t) & = & \big(m(X^\varepsilon (t), U(t)) - m(X (t), U(t)\big) dt \nonumber \\
	&& +  [ \sigma (X^\varepsilon (t)) - \sigma(X(t)) + \varepsilon \hat \sigma (Y^\varepsilon (t))] 
	dW(t) . 
\end{eqnarray}
Using  Theorem \ref{comparisonsde1}, 
we get 
$$P\big( Y^\varepsilon_1 (t) \leq Y^\varepsilon(t)  \leq  Y^\varepsilon_2 (t) \ 
{\rm for \ all}\, t \big)\,  = \, 1 .
$$ 

This implies 
\begin{equation}\label{comparisonsdes}
	\|Y^\varepsilon \| \leq \max \{\|Y^\varepsilon_1\|, \|Y^\varepsilon_2\|\} \ {\rm a.s} .
\end{equation}
Since, $h_R (\cdot)$ is a non decreasing function, using (\ref{eqlimit01}), we get 
\begin{eqnarray}\label{eqlimit2}
	I_1(R) & = & \frac{1}{t} E \int_0^t h_R ( \|Y^\varepsilon(s)\|) ds \nonumber \\
	& \leq & \max \Big\{ \frac{1}{t} E \int_0^t h_R ( \|Y^\varepsilon_1(s)\|) ds ,
	\frac{1}{t} E \int_0^t h_R ( \|Y^\varepsilon_2(s)\|) ds \Big\} \\ \nonumber
	& \leq &  \sum^2_{i=1} \frac{1}{t} E \int_0^t h_R ( \|Y^\varepsilon_i(s)\|) ds.
\end{eqnarray}
Also, 
\begin{eqnarray}\label{eqlimit3}
	I_2(R) & \leq   & 2 \|r\|_\infty \frac{1}{t} \int_0^t P(\max\{ |Y^\epsilon_1(s)|,|Y^\epsilon_2(s)|\}  \geq R)ds \nonumber \\
	&  \leq &  2\|r\|_\infty \sum_{i=1}^2 \frac 1t \int_0^t P(|Y^\epsilon_i(s)| \geq R)ds.
\end{eqnarray}

Let $\eta_{\epsilon,i}[u]$  be a limit point
of  empirical probability measures of  $Y^\epsilon_i(\cdot)$ satisfying
\begin{eqnarray*}
	\limsup_{t \to \infty} \frac{1}{t} E \Big[ \int^t_0 h_R( \|Y^\varepsilon_i(s) \|) ds \Big] 
	& = & \int h_R (\|x\|) \eta_{\epsilon,i}[u](dx) ,\\
	\limsup_{t \to \infty} \frac{1}{t}  \int^t_0 P \big(  \|Y^\varepsilon_i(s) \| \geq R \big) ds 
	& \leq  &  \eta_{\epsilon,i}[u](B^c_R).
\end{eqnarray*}

Hence, we have 
\begin{eqnarray}\label{eqlimit3.1}
	\limsup_{t \to \infty} I_1(R) & \leq  & \sum_{i=1}^2 \int_{\mathbb{R}^d} h_R(\|x\|) 
	\eta_{\epsilon, i}[u](dx), \\ \nonumber 
	\limsup_{t \to \infty} I_2(R) & \leq & 2 \|r\|_\infty \sum_{i=1}^2 \eta_{\epsilon, i}[u](B^c_R).
\end{eqnarray}
Using Lemma \ref{asymtoticbehavior}, we get,
\begin{eqnarray}\label{eqlimit4}
	\lim_{\varepsilon \to 0} \limsup_{t \to \infty} I_1(R) &\leq & \lim_{\varepsilon \to 0} 
	\sum_{i=1}^2 \int_{\mathbb{R}^d} h_R(\|x\|)  \eta_{\epsilon, i}[u](dx) \ = \ 0, \\ \nonumber
	\lim_{\varepsilon \to 0} \limsup_{t \to \infty} I_2(R) &\leq &  2 \|r\|_\infty \lim_{\varepsilon \to 0} \sum_{i=1}^2  \eta_{\varepsilon, i}[u](B^c_R)=0 \; \text{because} \; \delta_0(B^c_R)=0.
\end{eqnarray}
Thus,
\begin{eqnarray}\label{eqlimit5}
	\rho^*_{\varepsilon} (x)  - \rho (x, U(\cdot)) & \leq &  \liminf_{t \to \infty}
	\frac{1}{t} E_x \Big[ \int^t_0 r(X^{\varepsilon} (s), U(s)) ds \Big] \nonumber \\ \nonumber 
	&&  -  \, \liminf_{t \to \infty}  \frac{1}{t} E_x \Big[ \int^t_0 r(X (s), U(s)) ds \Big]  \\
	& \leq &  \limsup_{t \to \infty} \frac{1}{t} 
	E \Big[ \int^t_0 ( r(X^{\varepsilon} (s), U(s)) -  r(X (s), U(s))) ds \Big]  \\ \nonumber 
	& \leq & 
	\sum_{i=1}^2 \int_{\mathbb{R}^d} h_R(\|x\|)  \eta_{\varepsilon, i}[u](dx)
	+ 2 \|r\|_\infty  \sum_{i=1}^2  \eta_{\varepsilon, i}[u](B^c_R), R > 0 .
\end{eqnarray}
Hence using equations (\ref{eqlimit4}) and observing that $\rho(x, U(\cdot) ) = \infty$
if the the sde (\ref{statedegensde}) has no solution corresponding to the stationary Markov control $u (\cdot)$ with initial condition $x$, we arrive at the following  lemma.
\begin{lemma}\label{partialresult} Assume (A1) and  (B). 
	Then we have 
	\[
	\lim_{\varepsilon \to 0} \rho^*_\varepsilon (x) \ \leq \  
	\inf_{u(\cdot) \in  {\mathcal U}_{SM}} \rho(x, u(\cdot)) .
	\]
\end{lemma}

Set for $z = (y, x_1, x_2) \in \mathbb{R}^{3d}$, 
\begin{equation}\label{generatorZeps1}
	{\mathcal L}^\varepsilon_Z f (z) = \mathcal {L}_{\varepsilon, i} f + 
	\langle m(x_1, u( x_1)), \nabla_{x_1} f \rangle + \langle m(x_2, u(x_1)), 
	\nabla_{x_2} f \rangle + {\mathcal L}^1_{Z, \varepsilon} f(z),
\end{equation}
where ${\mathcal L}^1_{Z, \varepsilon}$ is given in (\ref{generatorZeps}).

\begin{lemma}\label{asymtoticbehavior1} Assume (A1) and  (B). 
	For $U (\cdot) \in {\mathcal U}$, given by $U(t) = u(  X^\varepsilon(t)), t \geq 0,$ 
	where $u(\cdot)$ is a smooth map such that 
	$(Y^\varepsilon_i(\cdot), X^\varepsilon (\cdot), X_\varepsilon(\cdot))$ are solutions of (\ref{auxiliarysdeY}), 
	(\ref{gen_perturb_sde}) and (\ref{statedegensde}) respectively.  Then
	any limit point $\eta_{\varepsilon_n , i}$ of the empirical measures of 
	$Y^{\varepsilon_n}_i (\cdot)$ converges weakly to $\delta_0$.
\end{lemma}
\begin{proof} Again, fix $i=1$ and  we denote $\mathcal{L}_{\epsilon, 1}$ by $\mathcal{L}_{\epsilon}$.\\
	Set $Z^\varepsilon (t) = (Y^\varepsilon_1 (t), X^\varepsilon (t), X_\varepsilon(t))^T, 
	t \geq 0$.  
	Then  the infinitesimal generator of $Z^\varepsilon (\cdot)$ is given by the differential operator 
	${\mathcal L}^\varepsilon_Z $ given in (\ref{generatorZeps1}).
	
	Using the smoothness of $u$ and  (A1), it follows that 
	$(X^\varepsilon (\cdot), X_\varepsilon (\cdot))$ converges in law to $(X(\cdot), X(\cdot))$ given by 
	\begin{equation}\label{Xsde11}
		d X(t) \  = \  m( X(t), u(X(t))) dt + \sigma (X(t)) d W(t).
	\end{equation}
	
	Consider the family of empirical measures $\mu^\varepsilon_t \in \mathcal{P}(\mathbb{R}^{3d})$ of
	$Z^\varepsilon (\cdot)$ given by
	\[
	\iiint_{\mathbb{R}^{3d}} f(z) \mu^\varepsilon_t (dz) \ = \ \frac{1}{t} E \Big[ 
	\int^t_0 f(Z^\varepsilon (s)) ds \Big], \ f \in C_b(\mathbb{R}^{3d}), \, t > 0 . 
	\]
	By repeating the arguments as in the proof of Lemma \ref{asymtoticbehavior},
	the tightness of the empirical measures $\{ \mu^\varepsilon_t : t > 0 \}$ follows for 
	each $\varepsilon > 0$. 
	
	Let $\mu^\varepsilon \in \mathcal{P}(\mathbb{R}^{3d})$ be a limit point of 
	$\{ \mu^\varepsilon_t : t \geq 0 \}$ for each $\varepsilon > 0$. Then it follows that 
	\begin{equation}\label{asymtoticbehavior1eq2}
		\iiint {\mathcal L}^{\varepsilon}_Z f (z) \mu^\varepsilon (dz) = 0 \ {\rm for\ all} \ f \in C^\infty_c (\mathbb{R}^{3 d}).
	\end{equation}
	Hence,  $\mu^\varepsilon$ is an invariant probability measure for $Z^\varepsilon (\cdot)$, see Appendix, Lemma \ref{invariantmeasurechar1}.
	
	Again  from
	(\ref{asymtoticbehavioreq1}), it follows that $\{ \mu^\varepsilon | \varepsilon > 0 \}$ is
	a tight family and let $\mu \in \mathcal{P}(\mathbb{R}^{3d})$ be a limit point of it.
	For each  $f \in C^\infty_c (\mathbb{R}^{3d})$, it is easy to see that 
	${\mathcal L}^{\varepsilon}_Z f \to {\mathcal L}^{0}_Z f $ uniformly, 
	where ${\mathcal L}^{0}_Z $  is the differential operator given by substituting $\varepsilon =0$ in 
	${\mathcal L}^{\varepsilon}_Z$. Now repeating the arguments as in the proof of 
	Lemma \ref{asymtoticbehavior}, i follows that $\eta_{\varepsilon, 1} \to \delta_0$ as 
	$\varepsilon \to 0$. This completes the proof. 
\end{proof}

\noindent {\bf (D)} Assume that
\[
\lim_{\varepsilon \to 0} \sup_{u \in {\mathcal  U}^{smooth}_{SM}}
\sup_{\eta_i \in  {\mathcal H}_{\varepsilon, i}[u] } 
\Big( \sum_{i=1}^2 \int_{\mathbb{R}^d} h_R(\|x\|) 
\eta_i(dx) + 2 \|r\|_\infty \sum_{i=1}^2 \eta_i (B^c_R)\Big)
\ = \ 0, 
\]
where ${\mathcal H}_{\varepsilon, i}[u]$ denote the set of all  invariant probability measure of 
$Y^\varepsilon_i(\cdot)$ given by (\ref{auxiliarysdeY}) corresponding to the stationary Markov
control $u(\cdot)$, $X^\varepsilon (\cdot) $ is given by (\ref{gen_perturb_sde}) correponding to 
the stationary Markov control $u(\cdot)$ and $X(\cdot)$ denote the solution of the sde 
(\ref{statedegensde}) corresponding to the control $u(X^\varepsilon(\cdot))$ and 
${\mathcal  U}^{smooth}_{SM}$ denote the set of all smooth stationary Markov controls
for (\ref{gen_perturb_sde}).

\begin{theorem}\label{main1} Assume (A1), (B) and (D). There exists a sequence of
	admissible  controls $U^\varepsilon(t) = u^\varepsilon (X^\varepsilon (t)), t \geq 0, 
	u^\varepsilon \in {\mathcal U}^{smooth}_{SM}$ satisfying 
	
	(i) $U^\varepsilon (\cdot)$ is $\varepsilon$-optimal for the problem (\ref{gen_perturb_sde});
	(\ref{ergodiccost-perturbed}) and
	
	(ii) 
	\[
	\lim_{\varepsilon \to 0} \rho_\varepsilon(x, U^\varepsilon (\cdot)) 
	= \inf_{U(\cdot) \in {\mathcal U}_{SM}} \rho(x, U(\cdot)), \ x \in \mathbb{R}^d.
	\]
	
\end{theorem}

\begin{proof} For $\varepsilon > 0$ and $\delta > 0$, define a smooth Markov control $u^{\varepsilon, \delta} $ as follows.
	\[
	u^{\varepsilon, \delta}(x) = u^*_\varepsilon \star \rho_\delta (x), \, x \in \mathbb{R}^d
	\]
	where $u^*_\varepsilon$ is a stationary optimal Markov control for the problem (\ref{gen_perturb_sde}); (\ref{ergodiccost-perturbed}), $\rho_\delta$ is the density of
	$N(0, \delta)$ and $\star$ denote  convolution product. Then
	\begin{equation}\label{eq1subelliptic}
		\lim_{\delta \to 0} \int f(v) u^{\varepsilon, \delta}(x) (dv) \ = 
		\ \int f(v) u^*_\varepsilon (x)(dv), \ f \in C(\mathbb{U}).
	\end{equation}
	Since ${\mathcal U}_{SM}$  is compact under the topology given in (\ref{BorkarTopology}), 
	along a subsequence $u^{\varepsilon, \delta } \to  u^\varepsilon $  in (\ref{BorkarTopology})
	as  $\delta \to 0$. Hence,  we have  for 
	$g \in C(\mathbb{U}), \, f \in  L^1(\mathbb{R}^d) \cap  L^2(\mathbb{R}^d)$, 
	\[ 
	\lim_{\delta  \to 0} \int_{\mathbb{R}^d}  f(x) \int_{\mathbb{U}} g(v) v^{\varepsilon, \delta}
	(x) (dv) dx = \int_{\mathbb{R}^d}  f(x) \int_{\mathbb{U}} g(v) u^\varepsilon (x) (dv) dx. 
	\]
	Hence we have
	\[
	\lim_{\delta \to 0} \int_{\mathbb{U}} g(v) u^{\varepsilon, \delta} (x) (dv) \ = \
	\int_{\mathbb{U}} g(v) u{^\varepsilon} (x) (dv)  \ {\rm a.e.} \ x \ {\rm for\ all}\ 
	g \in C(\mathbb{U}).
	\]
	Hence $u^\varepsilon = u^*_\varepsilon$, therefore  
	$u^{\varepsilon, \delta } \to  u^*_\varepsilon $ in ${\mathcal U}_{SM}$ as  $\delta \to 0$. 
	
	Now for each $\varepsilon > 0$ fixed, using Lemma 3.2.6, p.89, Lemma 3.3.4, pp.97-98,
	\cite{AriBorkarGhosh}, it follows that $\eta^{\varepsilon, \delta} \to \eta^*_\varepsilon$ in 
	$\mathcal{P}(\mathbb{R}^d)$, where $\eta^{\varepsilon, \delta}, \eta^*_\varepsilon$ denote 
	respectively the invariant measures of the processes (\ref{gen_perturb_sde}) corresponding to 
	the controls $u^{\varepsilon, \delta }$ and $  u^*_\varepsilon $. Hence there exists a smooth 
	control $u^\varepsilon$ such that 
	\begin{equation}\label{eq2subelliptic}
		|\rho_\varepsilon(x, u^\varepsilon (\cdot)) - \rho_\varepsilon (x, u^*_\varepsilon (\cdot))| 
		\leq \varepsilon.
	\end{equation}
	
	Let $X^{\varepsilon, \delta}(\cdot)$ denote the solution of the sde (\ref{gen_perturb_sde})
	corresponding to the smooth stationary Markov control $u^{\varepsilon, \delta}$ and 
	$X^\delta (\cdot)$ denote the solution of the sde (\ref{statedegensde}) corresponding to the 
	control  $U^{\varepsilon, \delta} (t) = u^{\varepsilon, \delta}(X^{\varepsilon, \delta} (t))$.
	Then we have 
	\begin{eqnarray*} 
		\frac{1}{t} E \Bigg[\int_0^t (r(X^{\varepsilon, \delta }(s),U^{\varepsilon, \delta}(s))
		-r(X^\delta (s), U^{\varepsilon, \delta}(s)))ds\Bigg]  & \\
		\ \leq  \ \frac{1}{t} E \Bigg[\int_0^t h_R(\|X^{\varepsilon, \delta}(s) - X^\delta(s)\|)ds\Bigg] 
		& \\
		+ 2\|r\|_\infty \frac{1}{t} \int_0^t P(\|X^{\varepsilon, \delta}(s) - X^\delta(s)\| \geq R)ds 
		& \\
		\ := I_1 (R) [u^{\varepsilon, \delta}] + I_2 (R) [u^{\varepsilon, \delta}]
	\end{eqnarray*}
	Let $\eta_{\epsilon,i}[u]$  be a limit point
	of  empirical probability measures of  $Y^\epsilon_i(\cdot)$ given by (\ref{auxiliarysdeY}) corresponding a smooth Markov control $u$ satisfying
	\begin{eqnarray*}
		\limsup_{t \to \infty} \frac{1}{t} E \Big[ \int^t_0 h_R( \|Y^\varepsilon_i(s) \|) ds \Big] 
		& = & \int h_R (\|x\|) \eta_{\epsilon,i}[u](dx) ,\\
		\limsup_{t \to \infty} \frac{1}{t}  \int^t_0 P \big(  \|Y^\varepsilon_i(s) \| \geq R \big) ds 
		& \leq  &   \eta_{\epsilon,i}[u](B^c_R).
	\end{eqnarray*}
	Hence,  it follows that 
	\begin{eqnarray}\label{Hormander_eqlimit1}
		\limsup_{t \to \infty} I_1(R)[u] & \leq &  \sum_{i=1}^2 \int_{\mathbb{R}^d} h_R(\|x\|) 
		\eta_{\epsilon, i}[u](dx), \nonumber \\ 
		\limsup_{t \to \infty} I_2(R)[u] & \leq & 2 \|r\|_\infty \sum_{i=1}^2 
		\eta_{\epsilon, i}[u](B^c_R).
	\end{eqnarray}
	Therefore,  
	\begin{eqnarray}\label{H_eqlimit2} 
		\limsup_{t \to \infty} 
		\frac{1}{t} E \Bigg[\int_0^t (r(X^{\varepsilon, \delta }(s),U^{\varepsilon, \delta}(s))
		-r(X^\delta (s), U^{\varepsilon, \delta}(s)))ds\Bigg] &    \\ \nonumber
		\ \leq \sup_{u \in {\mathcal  U}^{smooth}_{SM}} 
		\sup_{\eta \in {\mathcal H}_{\varepsilon, i}[u] }
		\Big( \sum_{i=1}^2 \int_{\mathbb{R}^d} h_R(\|x\|) 
		\eta(dx) + 2 \|r\|_\infty \sum_{i=1}^2  \eta(B^c_R)\Big). &
	\end{eqnarray}
	Using (\ref{H_eqlimit2}), we have 
	\begin{eqnarray}\label{H_eqlimit3}
		\rho^*_\varepsilon(x) - \inf_{U(\cdot) \in {\mathcal U}_{SM}} \rho(x, U(\cdot))
		& \geq & - \varepsilon - \rho_\varepsilon (x, U^{\varepsilon, \varepsilon} (\cdot) ) - 
		\rho (x, U^{\varepsilon, \varepsilon} (\cdot) ) \nonumber \\
		& \geq & - \varepsilon - \limsup_{t \to \infty} I_1(R)[u^{\varepsilon, \varepsilon}] - 
		\limsup_{t \to \infty} I_2(R)[u^{\varepsilon, \varepsilon}] \\ \nonumber 
		& \geq & - \varepsilon -  \sup_{u \in {\mathcal  U}^{smooth}_{SM}}
		\sup_{\eta \in {\mathcal H}_{\varepsilon, i}[u] }  
		\Big( \sum_{i=1}^2 \int_{\mathbb{R}^d} h_R(\|x\|) 
		\eta(dx) \\ \nonumber 
		&& \ \  \ \ + \, 2 \|r\|_\infty \sum_{i=1}^2  \eta(B^c_R)\Big). 
	\end{eqnarray}
	Now combining the above with Lemma \ref{partialresult} and the assumption (D), we complete the proof. 
\end{proof}
From the above Theorem, we can deduce the following.
\begin{theorem}\label{approximateoptimal} Assume (A1), (B) and (D). Then there exists a sequence
	of $\varepsilon$-ergodic controls $U^\varepsilon (\cdot)$ of for the problem(\ref{gen_perturb_sde});(\ref{ergodiccost-perturbed}) such that
	\[
	\lim_{\varepsilon \to 0} \rho(x, U^{\varepsilon} (\cdot) ) = 
	\inf_{u(\cdot) \in {\mathcal U}_{SM}} \rho(x, u(\cdot)), \ x \in \mathbb{R}^d.
	\]  
\end{theorem}
\begin{proof} Set $u^\varepsilon (\cdot) = u^{\varepsilon, \varepsilon}$, where 
	$u^{\varepsilon, \varepsilon}$ as in the proof of Theorem \ref{main1} and 
	$U^\varepsilon (t) = u^\varepsilon (X^\varepsilon (t)), t \geq 0$.  Then using 
	(\ref{H_eqlimit3}), we have
	\begin{eqnarray*}
		- \varepsilon - \sup_{u \in {\mathcal  U}^{smooth}_{SM}}
		\sup_{\eta \in {\mathcal H}_{\varepsilon, i}[u] } 
		\Big( \sum_{i=1}^2 \int_{\mathbb{R}^d} h_R(\|x\|) \eta(dx) 
		+ 2 \|r\|_\infty \sum_{i=1}^2  \eta(B^c_R)\Big) & \\
		\leq \ \rho^*_\varepsilon (x) - \rho(x, U^{\varepsilon} (\cdot)) & \\
		\ \leq \ \sup_{u \in {\mathcal  U}^{smooth}_{SM}} 
		\sup_{\eta \in {\mathcal H}_{\varepsilon, i}[u] } 
		\Big( \sum_{i=1}^2  \int_{\mathbb{R}^d} h_R(\|x\|) \eta(dx) 
		+ 2 \|r\|_\infty \sum_{i=1}^2   \eta(B^c_R)\Big).
	\end{eqnarray*}
	The last inequality folows from (\ref{eqlimit5}). 
	Now combining the above inequality with Theorem \ref{main1}, we get
	\[
	\lim_{\varepsilon \to 0} \rho(x, U^{\varepsilon} (\cdot) ) = 
	\inf_{u(\cdot) \in {\mathcal U}_{SM}} \rho(x, u(\cdot)), \ x \in \mathbb{R}^d.
	\]  
	
\end{proof}

\noindent Now we give a sufficient condition for the assumption (D).

\begin{lemma}\label{sufficientD} Assume $b_i (\cdot)$ in assumption (B) satisfies
	$\langle b_i(x), x \rangle <  - d^2 \|x \|^2 , x \in \mathbb{R}^d$ 
	and  Lipschitz constant $L$ of $\sigma_{ij} $ for all $i, j$ satisfies $d^2 L^2 < 2$,  then (D) holds.
\end{lemma}

\begin{proof} Set $\hat a = \hat \sigma(0) \hat \sigma(0)^T$. 
	For $u \in {\mathcal U}_{SM}$, let $Y^\varepsilon_i (\cdot)$
	denote the solution to (\ref{auxiliarysdeY}) corresponding to $u(X^\varepsilon (\cdot))$. 
	Then using Ito's formula, we have
	\begin{eqnarray*}
		d E  \|Y^\varepsilon_i(t)\|^2 & = & \Big( 2 E \langle b_i( Y^\varepsilon_i (t)), Y^\varepsilon_i(t) 
		\rangle + \frac{1}{2}\sum^d_{j, k=1} E (\sigma_{jk} (X^\varepsilon (t)) 
		- \sigma_{jk}(X_\varepsilon (t)))^2\\ 
		&& + \varepsilon^2 \sum^d_{j, k=1} \hat a_{jk} \Big) dt \\
		& \leq & \Big( - 2 E \| Y^\varepsilon_i(t) \|^2 
		+ \frac{1}{2} d^2 L^2 \| Y^\varepsilon (t) \|^2  + \varepsilon^2 
		\sum^d_{j, k=1} \hat a_{jk} \Big) dt .
	\end{eqnarray*}
	Set $$f(t) = E  [\|Y^\varepsilon_1(t)\|^2 + \|Y^\varepsilon_2(t)\|^2], t \geq 0.$$ 
	Then we get
	\begin{eqnarray*}
		d f(t) & \leq & (- 2 f(t) + d^2 L^2 E \|Y^\varepsilon(t)\|^2 + \sum^d_{j, k=1} \hat a_{jk} ) dt\\
		& \leq & -( 2 - d^2 L^2) f(t) dt + \sum^d_{j, k=1} \hat a_{jk}  dt . \\
	\end{eqnarray*}
	The second inequality follows from Theorem \ref{comparisonsde1}. 
	Now using  comparison theorem of odes, we get 
	\begin{equation}\label{estimate1D}
		E  [\|Y^\varepsilon_1(t)\|^2 + \|Y^\varepsilon_2(t)\|^2] \leq 
		\frac{\varepsilon^2 }{2 - d^2 L^2} \big(1 - e^{- (2 - d^2 L^2)t} \big)
		\sum^d_{j, k=1} \hat a_{jk}, t \geq  0 .
	\end{equation}
	Now using Chebychev's inequality, we get
	\[
	P( \| Y^\varepsilon_i(t) \| \geq R) \leq 
	\frac{\varepsilon^2 }{(2 - d^2 L^2) R^2} \big(1 - e^{- (2 - d^2 L^2)t} \big) 
	\sum^d_{j, k=1} \hat a_{jk}  . 
	\]
	Hence
	\begin{equation}\label{estimate2D}
		\frac{1}{t} \int^t_0 P(\|Y^\varepsilon_i (s) \| \geq R ) ds  \ \leq 
		\frac{\varepsilon^2}{(2 - d^2 L^2) R^2} 
		\sum^d_{j, k=1} \hat a_{jk} := \frac{K_1 \varepsilon^2}{R^2} , t > 0 .
	\end{equation}
	From (\ref{estimate2D}), we get for $u \in {\mathcal U}_{SM}$, 
	\begin{equation}\label{estimate3D}
		\eta  (B^c_R) \leq \frac{K_1 \varepsilon^2}{R^2}, \ \eta \in 
		{\mathcal H}_{\varepsilon , i} [u] . 
	\end{equation}
	Hence for $f \in C_b(\mathbb{R}^d)$,
	\begin{eqnarray}\label{estimate4D}
		\sup_{u \in {\mathcal U}_{SM}} \sup_{\eta \in {\mathcal H}_{\varepsilon , i} [u] } 
		\int f(x) \eta  (dx) 
		& = & \sup_{u \in {\mathcal U}_{SM}} \sup_{\eta \in {\mathcal H}_{\varepsilon , i} [u] } 
		\int_{\|x\| \leq \sqrt{\varepsilon} }   f(x) \eta (dx) \nonumber \\ \nonumber 
		&& + \sup_{u \in {\mathcal U}_{SM}} \sup_{\eta \in {\mathcal H}_{\varepsilon , i} [u] } 
		\int_{\|x\| >  \sqrt{\varepsilon} }  f(x) \eta  (dx)\\
		& \leq & \sup_{\|x \| \leq \sqrt{\varepsilon}} |f(x) | + \|f\|_\infty K_1 \varepsilon \\ \nonumber 
		& \to & 0 \ {\rm as } \ \varepsilon \to 0 .
	\end{eqnarray}
	Hence (D) follows.

\end{proof}

\begin{theorem}\label{errorbound1} Assume (A1), (B1)(i), (C) and 
	$\langle b_i(x), x \rangle <  - d^2 \|x \|^2 , x \in \mathbb{R}^d$ 
	and  Lipschitz constant $L$ of $\sigma_{ij} $ for all $i, j$ satisfies $d^2 L^2 < 2$ Then
	\[
	|\rho^*_\varepsilon (x) - \hat \rho^*_{{\mathcal U}_{SM}} (x) | \ \leq \ \hat K 
	\sqrt{\varepsilon}, \, x \in \mathbb{R}^d, 
	\]
	where $\hat K > 0$ is some constant. 
	
\end{theorem}

\begin{proof} 
	Note that given hypothesis implies (B). Hence from Lemma \ref{partialresult},
	and using (\ref{estimate3D}) and (\ref{estimate4D}), we get 
	\begin{eqnarray*}
		\rho^*_\varepsilon (x) - \rho^*_{{\mathcal U}_{SM}} (x) & \leq & 
		\sup_{u \in {\mathcal  U}_{SM}} \sup_{\eta_i \in {\mathcal H}_{\varepsilon , i} [u] }
		\Big( \sum_{i=1}^2 \int_{\mathbb{R}^d} h_R(\|x\|) \eta_i (dx) 
		+ 2 \|r\|_\infty \sum_{i=1}^2  \eta_i (B^c_R)\Big)\\
		& \leq & 2 K_1 \varepsilon^2 + 2 \sup_{\|x\| \leq \sqrt{\varepsilon}} h_1 (\|x\|) + 
		2 \|h_1\|_\infty K_1 \varepsilon  + 4 \|r \|_\infty \frac{K_1 \varepsilon^2}{R^2} \\
		& = & 2 \big(1 + \frac{2 \|r\|_\infty}{R^2} \big)  K_1 \varepsilon^2 + 2 {\rm Lip}(r) \sqrt{\varepsilon} 
		+ 2 {\rm Lip}(r) K_1 \varepsilon .
	\end{eqnarray*}
	From (\ref{eq2subelliptic}), (\ref{H_eqlimit2}), we get
	\begin{eqnarray*}
		\rho^*_\varepsilon (x) - \rho^*_{{\mathcal U}_{SM}} (x) & \geq & 
		- \sup_{u \in {\mathcal  U}_{SM}}\sup_{\eta_i \in {\mathcal H}_{\varepsilon , i}[u]}
		\Big( \sum_{i=1}^2 \int_{\mathbb{R}^d} h_R(\|x\|)  \eta_i(dx) 
		+ 2 \|r\|_\infty \sum_{i=1}^2  \eta_i(B^c_R)\Big) - \varepsilon \\
		& \geq  & - 2 \big(1 + \frac{2 \|r\|_\infty}{R^2}\big) K_1 \varepsilon^2 - 2 {\rm Lip}(r) \sqrt{\varepsilon} - 2 {\rm Lip}(r) K_1 \varepsilon - \varepsilon .
	\end{eqnarray*}
	Combining the above two inequalities, the result  follows. 
	
\end{proof}

\begin{remark}\label{comment_optimalcontrol} Neither Theorem \ref{main1} nor 
	Theorem \ref{approximateoptimal} implies that 
	$\rho^*_{\mathcal U} (x) = \rho^*_{{\mathcal U}_{SM}} (x)$ for all $x \in \mathbb{R}^d$. 
	But from Theorem \ref{main2}, it follows that $\rho(x, U^*(\cdot)) = \rho^*_{\mathcal U} (x)$
	for all $x \in {\rm supp}(\eta^*)$.\\ i.e., $U^*(\cdot)$ is ergodic optimal for all initial 
	conditions $x \in {\rm supp}(\eta^*)$, in the sense of Definition \ref{ergodicoptimal1}(i).
\end{remark}

\subsection{Constant Diffusion matrix} We consider the case when $\sigma$ is a constant non negative definite matrix. Set $\hat \sigma (0) = \hat \sigma$. The sdes 
(\ref{gen_perturb_sde}), (\ref{auxiliarysdeY}) and (\ref{sdeY}) respectively takes the forms
\begin{eqnarray}
	d X^\varepsilon (t) & = & m(X^\varepsilon (t), U(t)) dt + ( \sigma + \varepsilon \hat \sigma ) 
	d W(t), \label{newYsde1} \\
	d Y^\varepsilon_i(t) & = & b_i (Y^\varepsilon_i (t)) dt + \varepsilon \hat \sigma d W(t), 
	\label{newYsde2}\\
	d Y^\varepsilon (t) & = & (m(X^\varepsilon (t), U(t)) - m(X(t), U(t))) dt + \varepsilon \, 
	\hat \sigma dW(t).\label{newYsde3}
\end{eqnarray} 

Let $V_i(0, x)$ denote the quasi potential given by
\begin{eqnarray*}
	V_i(0, x) & =  & \inf \Big\{ \frac{1}{2} \int^T_0 \langle (\dot{\varphi} (t) - b_i(\varphi (t)) , \hat a (\varphi (t)) (\dot{\varphi} (t) - b_i(\varphi (t)) \rangle
	\, dt :  \\
	&& \varphi \in C[0, \ T], \varphi(0) =0 , \varphi (T) = x , T > 0 \Big\} \\
	& = & \inf_{u(\cdot) \in \mathcal{U}_i} \int^\infty_0 u(t)^T \hat{a}^{-1} u(t) dt 
\end{eqnarray*}
where $\hat a = \hat \sigma \hat \sigma^T$ and $\mathcal{U}_i$ is the set of all 
measurable $u : [0, \ \infty) \to \mathbb{R}^d$ such that the solution $y_i(\cdot)$ of 
\[
d y_i(t) =  - (b_i (y_i(t)) + u(t)) dt, y(0) = x
\]
satisfies $y_i(t) \to 0$ as $t \to \infty$. 

Under (B), the sde (\ref{newYsde2}) has a unique invariant probability distribution. Since the process $Y^\varepsilon_i(\cdot)$ is Feller, it follows that any limit point  
$\eta_{\varepsilon, i}$ of empirical measures given in  Lemma \ref{asymtoticbehavior} is an invariant distribution, see for example Theorem 1.5.15, p.22, \cite{AriBorkarGhosh}, and hence empirical 
measures converges to its unique invariant probability distribution $\eta_{\varepsilon, i}$.  
Therefore  from Lemma \ref{asymtoticbehavior}, we have
$\eta_{\varepsilon, i} \to \delta_0$ weakly. 

\noindent We use the following large deviations result about the invariant probability measures of  (\ref{newYsde2}).  
\begin{lemma}\label{FreidlinWentezellresult2} Assume (C). The process $Y^{\varepsilon}_i (\cdot)$ given by (\ref{newYsde2})
	has a unique invariant probability measure $\eta_{\varepsilon , i}$ has a density 
	$\varphi_{\varepsilon, i}$ 
	and satisfies 
	
	(i)
	\[
	\lim_{\varepsilon \to 0} \varepsilon^2 \ln \eta_{\varepsilon ,i} (D) = - \inf_{x \in D} V_i(0, x) , \ i =1,2 , 
	\]
	for any domain $D \subseteq \mathbb{R}^d$.
	
	(ii) 
	\[
	\lim_{\varepsilon \to 0} \varepsilon^2 \ln \varphi_{\varepsilon , i} (x) = - V_i(0, x), x \in 
	\mathbb{R}^d.
	\]
\end{lemma}

\begin{proof} The first part  follows easily by mimicking the arguments of Theorem 4.3, pp.111-113, \cite{FreidlinWentzell}. 
	The second part follows from Theorem 2, \cite{AnupBorkar}.
	
\end{proof}

\begin{theorem}\label{errorbound2}
	Assume (A1), (B) and (C) and that $\sigma$ is a constant non negative definite matrix.  
	
	\noindent (i) The following inequality holds: 
	$$|\rho^*_\epsilon(x) - \rho^*_{\mathcal U}(x)| \leq 2 \int_{\mathbb{R}^d} h_1(\|x\|) 
	e^{- \sum^2_{i=1}\frac{V_i(0, x) }{\varepsilon^2}} dx + 
	4 \|r\|_\infty \sum^2_{i=1} e^{- \frac{\inf_{\|x \| \geq 1} V_i(0, x) }{\varepsilon^2}} .
	$$   
	\noindent (ii) $\rho^*_{{\mathcal U}_{SM} } (x) = \rho^*_{\mathcal U}(x)$ for all $x \in \mathbb{R}^d$.
	
	\noindent (iii) If $U^*_\varepsilon(\cdot)$ denote the ergodic optimal control of  
	(\ref{newYsde1});(\ref{ergodiccost-perturbed}), then
	\begin{eqnarray*}
		\lim_{\varepsilon \to 0} \rho(x, U^*_\varepsilon (\cdot)) & = & \rho^*_{\mathcal U} (x) ,
		x \in \mathbb{R}^d, \\
		\lim_{\varepsilon \to 0} \rho_\varepsilon (x, U^*_\varepsilon (\cdot)) & = & \rho^* .
	\end{eqnarray*}
	
\end{theorem}

\begin{proof} For $x \in \mathbb{R}^d$, an admissible pair $(X(\cdot), U(\cdot))$ of
	the sde (\ref{statedegensde}) such that $X(0)= x$, we can assume without any loss of generality 
	that $U(\cdot)$ is in prescribed form. Let  $X^\varepsilon (\cdot)$ respectively denote 
	solutions to (\ref{statedegensde})  corresponding to the  control 
	$U(\cdot)$ and initial condition $x$. Then 
	\begin{eqnarray*}
		\rho^*_\varepsilon (x) - \rho(x, U(\cdot)) & \leq  & \rho_\varepsilon(x, U(\cdot) - 
		\rho(x, U(\cdot)) \\
		& \leq & \limsup_{t \to \infty} \frac{1}{t} 
		E \Bigg[\int_0^t (r(X^\epsilon(s),U(s))-r(X(s), U(s)))  ds\Bigg]  \\
		& \leq & \limsup_{t \to \infty} \frac{1}{t}
		E \Bigg[\int_0^t h_R(\|X^\epsilon(s) - X(s)\|)ds\Bigg] \\
		&& + 2\|r\|_\infty \limsup_{t \to \infty}  \frac{1}{t} \int_0^t P(\|X^\epsilon(s) - X(s)\| \geq R)ds. 
	\end{eqnarray*}
	
	Now using the same argument as in the proof of Lemma \ref{partialresult}, we have 
	\[
	\rho^*_\epsilon(x) - \rho^*_{\mathcal U}(x) \leq  \sum^2_{i=1} \int_{\mathbb{R}^d} h_1(\|x\|)\eta_{\epsilon, i} (dx) + 2\|r\|_{\infty} \sum_{i=1}^2 \eta_{\epsilon, i}(B^c_1),
	\]
	where $\eta_{\epsilon, i} (dx)$ is the unique invariant probability measure of 
	$Y^\varepsilon_i (\cdot)$. Now repeat the argument with $U^\varepsilon (t) = u^*_\varepsilon 
	(X^\varepsilon (t)), \, t \geq 0,$ where $u^*_\varepsilon $ is an ergodic optimal stationary Markov
	control of (\ref{newYsde1});(\ref{ergodiccost-perturbed}), we get
	\[
	\rho^*_\epsilon(x) - \rho^*_{\mathcal U}(x) \geq - \Big( 
	\sum^2_{i=1} \int_{\mathbb{R}^d} h_1(\|x\|)\eta_{\epsilon, i} (dx) + 2\|r\|_{\infty} \sum_{i=1}^2 \eta_{\epsilon, i}(B^c_1) \Big).
	\]
	Thus we have 
	\begin{equation}\label{eq1errorbound} 
		|\rho^*_\epsilon(x) - \rho^*_{\mathcal U}(x)| \leq  \sum^2_{i=1} \int_{\mathbb{R}^d} h_1(\|x\|)\eta_{\epsilon, i} (dx) + 2\|r\|_{\infty} \sum_{i=1}^2 \eta_{\epsilon, i}(B^c_1).
	\end{equation}
	From Lemma \ref{FreidlinWentezellresult2} (i), we get for small enough $\varepsilon$
	\begin{equation}\label{eq1errorbound1} 
		0 \leq \eta_{\varepsilon , i} (B^c_1) \leq 2 
		e^{- \frac{\inf_{\|x \| \geq 1} V_i(0, x) }{\varepsilon^2}}, \ i =1, 2.
	\end{equation}  
	From Lemma \ref{FreidlinWentezellresult2} (ii), we get for small enough $\varepsilon$
	\begin{equation}\label{eq2errorbound1} 
		0 \leq \int_{\mathbb{R}^d} h_1(\|x\|) \eta_{\varepsilon , i} (dx)  \leq 2 
		\int_{\mathbb{R}^d} h_1(\|x\|) e^{- \frac{V_i(0, x) }{\varepsilon^2}} dx , \ i =1, 2.
	\end{equation}
	Now combining (\ref{eq1errorbound}), (\ref{eq1errorbound1}) and (\ref{eq2errorbound1}), 
	we get the desired bounds(i). Now proof of (ii) is easy and the proof of second limit in (iii) follows from Theorem \ref{main2}.
	

\end{proof} 

\section{Tunneling of controlled gradient flows in $\mathbb{R}$} 
The assumption (B) in Section 4, implicitly enforces  that $m(\cdot, \cdot)$  can not have multiple local extrema for any control strategy.  Hence it is not suprising to see that $\rho_\varepsilon (x, U^\varepsilon (\cdot)) \to \rho^*_{\mathcal U} (x)$ for each $x \in \mathbb{R}^d$. 
This will not be the case in general which is the subject matter of this section. 
In this section, we establish tunneling for the state dynamics under optimal stationary Markov control  when the controlled dynamics is given by 
deterministic gradient flow in  $\mathbb{R}$  with an additive control taking values in 
$\mathbb{U}$ which is a  closed and bounded interval. 

Let $V : \mathbb{R} \to \mathbb{R}$ be a smooth
function with finitely many points of extrema and satisfies the growth condition
\[
\lim_{ |x | \to \infty} \frac{V(x)}{|x|^{ 1 + \beta} } \ = \ \infty,
\] 
for some $\beta > 0$. 

We consider ergodic control problem with state dynamics 
\begin{equation}\label{stateode-tunneling}
	d X(t) =  \big( -V' (X(t)) + U(t) \big) dt
\end{equation}
and cost criterion 
\begin{equation}\label{cost-tunneling}
	\rho(x, U(\cdot)) \ = \ \liminf_{t \to \infty} \frac{1}{t} 
	E_x \Big[ \int^t_0 r(X(s), U(s)) ds \Big],
\end{equation}
where $U(\cdot) \in {\mathcal U}$ and $X(\cdot)$ is the solution to (\ref{stateode-tunneling})
with initial condition $X(0) =x$. 
The corresponding small noise perturbed ergodic optimal control problem is given with 
state dynamics 
\begin{equation}\label{statesde-tunneling}
	d X^\varepsilon(t) =  \big( -V' (X^\varepsilon(t)) + U(t) \big) dt + \varepsilon d W(t)
\end{equation}
and cost criterion 
\begin{equation}\label{smallnoisecost-tunneling}
	\rho_\varepsilon(x, U(\cdot)) \ = \ \liminf_{t \to \infty} \frac{1}{t}
	E_x \Big[ \int^t_0 r(X^\varepsilon(s), U(s)) ds \Big],
\end{equation}
where $U(\cdot) \in {\mathcal U}$ and $X^\varepsilon(\cdot)$ is the solution to 
(\ref{statesde-tunneling}) with initial condition $x$. 

Let $u^\varepsilon (\cdot)$ be an optimal  stationary Markov control  for 
(\ref{statesde-tunneling}); (\ref{smallnoisecost-tunneling}). We make the following
assumption about $u^\varepsilon (\cdot)$. 

We assume:
\begin{assumption*}[\bf T] (i) The potential $V_\varepsilon $ given by
	\begin{equation}\label{potentialVepsilon}
		V_\varepsilon (x) = \ V(x) - \int^x_0 u^\varepsilon (y) dy, x \in \mathbb{R}
	\end{equation}
	has finitely many critical points .
	
	(ii) Along a subsequence  $u^\varepsilon$ converges pointwise and we denote  its limit point
	by $u^0$. 
	
	(iii) The potential $V_0 $ given by
	\begin{equation}\label{potentialVepsilon1}
		V_0 (x) = \ V(x) - \int^x_0 u^0 (y) dy, x \in \mathbb{R}
	\end{equation}
	has finitely many critical points.
	
	(iv) There exists $\delta$-ergodic optimal $u^\delta$ for (\ref{stateode-tunneling}); 
	(\ref{cost-tunneling}),  which is smooth and  $u^\delta \to u^0$ pointwise.
	
	(v) $u^\varepsilon$ is piecewise smooth and $(u^\varepsilon)'_{+}$ is uniformly bounded
	on compact sets.

\end{assumption*}

For the remaining analysis, we take the subsequence as $\varepsilon$ itself. Observe that $V_\varepsilon \to V_0$ uniformly on compact sets and hence any extremum of $V_\varepsilon$
converges to the corresponding local extremum of $V_0$. Hence without any loss of generality,
we assume that $V_\varepsilon, V_0$ have the same set of local extrema given by 
$x_1 < y_1 < x_2 <  \cdots <  y_{N-1} < x_N$ where $y_i$'s are local maxima and $x_i$'s 
represent local minima.  
Set
\[
E_1 \ = \ (-\infty, \ y_1 ), E_N = (y_{N-1}, \ \infty), \ E_i \ = \ (y_{i-1}, \ y_i), 
\ i = 2, \cdots , N- 1.
\]

For $\delta > 0$, set
\begin{equation}\label{continuousapproximation1}
	u^{\varepsilon}_\delta (x) \ = \ \frac{1}{\delta} \int^{x+\delta}_{x} u^\varepsilon
	(y) dy , \
	u^{0}_\delta (x) \ = \ \frac{1}{ \delta} \int^{x+\delta}_{x} u^0 (y) dy . 
\end{equation}
Using the arguments for the proof of (\ref{eq2subelliptic}) in Theorem \ref{main1}, we have 
\begin{equation}\label{approximatevalue-eps}
	\rho_\varepsilon (x, u^\varepsilon_\delta (\cdot)) \ = \ 
	\rho_\varepsilon (x, u^\varepsilon (\cdot)) + O(\delta) .
\end{equation}
We will use the following facts. For a piecewise smooth  function $f : \mathbb{R} \to 
\mathbb{R}$, 
\begin{itemize}
	\item If $f$ is right continuous, then $f_\delta$ is differentiable  and 
	\[
	f'_\delta (x) \ = \ \frac{1}{ \delta} ( f(x+ \delta +) - f(x +)),
	\]
	where $f(x+)$ denote the right limit of $f$ at $x$.
	\item More over 
	\[
	\lim_{\delta \to 0} f'_\delta (x) = f'_{+}(x),
	\]
	where $f'_+(x)$ denote the right derivative of $f$ at $x$.
	
\end{itemize}
Define the potentials 
\begin{eqnarray}\label{continuousapproximation2}
	V_{\varepsilon, \delta} (x) & = &  V(x) - \int^x_0 u^\varepsilon_\delta (y) dy , \\ \nonumber
	V_{0, \delta} (x) & = &  V(x) - \int^x_0 u^0_\delta (y) dy, x \in \mathbb{R}
\end{eqnarray}
It is easy to see that for $\delta > 0$ small enough, set of local extrema of 
$V_{\varepsilon, \delta}, V_{0, \delta}$ is $\{x_1, y_1, \cdots , y_{N-1}, x_N \}$.
Also using assumption (T), it follows that 
\begin{eqnarray}
	|V_{\varepsilon, \delta }(x) - V_\varepsilon (x) | & \leq & K_1 (a, b) \delta , 
	\ x \in [a, \ b], \label{moduluscontinuity1} \\
	|V_{\varepsilon, \delta }(x) - V_{0, \delta} (x) | & \leq & \omega_{a, b} (\varepsilon) , 
	\ x \in [a, \ b], \label{moduluscontinuity2} \\
	\lim_{\varepsilon \to 0} V''_{\varepsilon, \delta}(x) & = & V''_{0, \delta}(x), \, x \in 
	\mathbb{R}, \label{2derivativecontinuity}
\end{eqnarray}
where $K_1 (a, b)$ is a constant depends only on $a, b$ and $\omega_{a, b}$ is a modulus of
continuity. Set
\begin{equation}\label{depth}
	\Bar{V}_{\varepsilon, \delta}(y_{j} ) = V_{\varepsilon, \delta} (y_{j}) - 
	V_{\varepsilon, \delta} (x_i), \ j = i-1, i, \varepsilon \geq 0.
\end{equation}
Consider the ode 
\begin{equation}\label{delta-ode}
	d X (t) \ = \ - V'_{0,\delta} (X(t)) dt 
\end{equation}
and the small noise perturbed sdes given by
\begin{equation}\label{delta-smallnoise_sde}
	d X^\varepsilon_\delta (t) \ = \ - V'_{\varepsilon, \delta} (X^\varepsilon_\delta(t)) dt 
	+ \varepsilon d W(t) 
\end{equation}
and 
\begin{equation}\label{smallnoise_sde}
	d X^\varepsilon (t) \ = \ - V'_{\varepsilon} (X^\varepsilon(t)) dt + \varepsilon
	d W(t),
\end{equation}
where $W(\cdot)$ is a $\{{\mathcal F}_t \}$-Wiener process in $\mathbb{R}$. 
Define the stoping times $\tau^{\varepsilon, \delta}_i, \tau^{\varepsilon}_i, 
i = 1, 2, \cdots , N$ as 
\begin{equation}\label{exittime}
	\tau^{\varepsilon, \delta}_i \ = \ \inf \{ t \geq 0 | X^\varepsilon_\delta (t) \notin E_i \},
\end{equation}

\begin{equation}\label{exittime1}
	\tau^{\varepsilon}_i \ = \ \inf \{ t \geq 0 | X^\varepsilon (t) \notin E_i \}.
\end{equation}

\begin{lemma}\label{exittimeestimate} For each $\delta > 0$, follow holds.
	\[
	\lim_{\varepsilon \to 0} \varepsilon^2 \ln E_x \tau^{\varepsilon, \delta}_i \ = \ 
	\min \{ \Bar{V}_{0, \delta} (y_{i-1}) , \Bar{V}_{0, \delta} (y_i) \} := 2 \lambda^\delta_i, 
	x \in E_i, i =1, 2, \cdots , N .
	\]

\end{lemma}
\begin{proof} Observe that $V'_{\varepsilon, \delta} \to V'_{0, \delta}$ uniformly on compact
	subsets and $V''_{\varepsilon, \delta} \to V''_{0, \delta}$ pointwise. By closely mimicking the 
	arguments in the proof of Theorem 4.1, pp.106-109 of \cite{FreidlinWentzell}, it follows
	that
	\[
	\lim_{\varepsilon \to 0} \Big( \varepsilon^2 \ln E_x \tau^{\varepsilon, \delta}_i \, - \,  
	\min \{ \Bar{V}_{\varepsilon, \delta} (y_{i-1}) , \Bar{V}_{\varepsilon, \delta} (y_i) \} \Big)
	= 0.
	\]
	Hence the  result follows.

\end{proof}
Let 
\begin{eqnarray}\label{deepwell}
	S = \ {\rm argmax}_i \lambda^\delta_i & := & \{x_{m_1}, x_{m_2}, \cdots , x_{m_\kappa} \},
	\  x_{m_1} < x_{m_2} <  \cdots < x_{m_\kappa}, \nonumber \\
	\lambda^\delta & = & \max_i \lambda^\delta_i, \\ \nonumber 
	W_1 & = & (-\infty, \ y_{m_1}), \ W_\kappa \ = \ (y_{m_{\kappa -1}}, \ \infty), \\ \nonumber 
	W_i & = & (y_{m_{i-1}}, \ y_{m_i}), \ i = 2, \cdots, \kappa 
\end{eqnarray}
denote the deep wells and their respective local minima. 
Define $\hat X^\varepsilon_\delta (t) = X^\varepsilon_\delta (e^{\frac{2}{\varepsilon^2} 
	\lambda^\delta} t), t \geq 0$. 
Let $(y^1_i, \ y^2_i) = V_i \subset V'_i \subset W_i, \, i =1,2, \cdots, \kappa$ be such that
$x_{m_i} \in V_i$ for all $i$ and $V'_i$ is a distance $r> 0$ from the boundary $\partial W_i$.
Set $V =  \cup^\kappa_{i=1} V_i$. 
Define the trace of $\hat X^\varepsilon_\delta (\cdot)$ as follows.
\begin{eqnarray}\label{trace}
	Y^\varepsilon_\delta (t) & = & \hat X^\varepsilon_\delta (S^{\varepsilon, \delta} (t)), 
	t \geq 0, \ \  {\rm where}  \\ \nonumber 
	T^{\varepsilon, \delta} (t) & = & \int^t_0 I_{ \{\hat X^\varepsilon_\delta (s) \in V \} }
	ds, t \geq 0, \\ \nonumber 
	S^{\varepsilon, \delta}(t)& = & \sup \{ s \geq 0 : T^{\varepsilon, \delta} (s) \leq t \},
	t \geq 0.
\end{eqnarray}
It is easy to see that for each $t$, $S^{\varepsilon, \delta}(t)$ is a $\{ {\mathcal F}_t \}$-
stopping time.
Set  ${\mathcal G}_t = {\mathcal F}_{S (t)}, t \geq 0$. 

Let $\Psi : \mathbb{R} \setminus \{y_{m_1}, \cdots ,
y_{m_\kappa} \} \to S$ be defined as
\[
\Psi (x) = x_{m_i}, \ {\rm if} \ x \in W_i, i =1, 2, \cdots , \kappa.
\]
Consider the process
\[
Z^\varepsilon_\delta (t) = \Psi( Y^\varepsilon_\delta (t)) , t \geq 0.
\]

To characterize the limit of $Z^\varepsilon_\delta (\cdot)$, we need some auxillary results.
Set
\begin{equation}\label{Cdelta}
	C_\delta \ = \ \sum^\kappa_{j=1} \frac{1}{\sqrt{ | V''_{0, \delta} (x_{m_j})|}}.
\end{equation}

\begin{lemma}\label{normalizingconstant} Follwing estimates holds.
	
	(i) 
	\[
	\int_\mathbb{R} e^{-\frac{2}{\varepsilon^2} V_{\varepsilon, \delta} (x) } dx 
	\ = \ \sqrt{\pi } \varepsilon e^{-\frac{2}{\varepsilon^2}\lambda^\delta} C_\delta
	(1 + O(\varepsilon)) .
	\] 
	(ii) Let $\eta^{\varepsilon, \delta} $ denote the invariant distribution of
	(\ref{delta-smallnoise_sde}). Then 
	\begin{equation}\label{invariantblocks}
		\eta^{\varepsilon, \delta} (V_i) = \frac{1}{ C_\delta \, \sqrt{ | V''_{0, \delta} 
				(x_{m_i})|}} \big( 1 + O(\varepsilon)\big), i =1, 2, \cdots , \kappa .
	\end{equation}
\end{lemma}
\begin{proof} Applying Laplace method, we get
	\begin{eqnarray*}
		\int_\mathbb{R} e^{-\frac{2}{\varepsilon^2} V_{\varepsilon, \delta} (x) } dx 
		& = & \sqrt{\pi } \varepsilon e^{-\frac{2}{\varepsilon^2}\lambda^\delta} 
		\sum^\kappa_{j=1} \frac{1}{\sqrt{ | V''_{\varepsilon, \delta} (x_{m_j})|}} (1 + O(\varepsilon))\\
		& = & \sqrt{\pi } \varepsilon e^{-\frac{2}{\varepsilon^2}\lambda^\delta} 
		\sum^\kappa_{j=1} \frac{1}{\sqrt{ | V''_{0, \delta} (x_{m_j})|}} (1 + O(\varepsilon)).
	\end{eqnarray*}
	The last equality follows from (\ref{2derivativecontinuity}). 
	This completes the proof of (i) . 
	
	A similar application of Laplace method to the interval $V_i$ yields (ii).

\end{proof}

Let $\hat \tau^{\varepsilon, \delta}_i $ denote the exit  time of the process 
$X^\varepsilon_\delta (\cdot)$ from $W_i, \, i =1, 2, \cdots, \kappa$.
Set
\begin{eqnarray}\label{transitionprob}
	P_{x_{m_i}}\big( X^\varepsilon_\delta (\hat \tau^{\varepsilon, \delta}_i) = y_{m_i}\big) & = & p^{\varepsilon, \delta}_{i i +1}, \nonumber \\ \nonumber 
	P_{x_{m_i}}\big( X^\varepsilon_\delta (\hat \tau^{\varepsilon, \delta}_i) = y_{m_{i-1}}\big) & = & p^{\varepsilon, \delta}_{i i -1},\\
	p^\delta_{i i +1} & = & \frac{\sqrt{|V''_{0, \delta} (y_{m_{i-1}})|}}{\sqrt{|V''_{0, \delta}
			(y_{m_{i-1}})|} + \sqrt{|V''_{0, \delta} (y_{m_i})|}}\\ \nonumber 
	p^\delta_{i i -1} & = & \frac{\sqrt{|V''_{0, \delta} (y_{m_i})|}}{\sqrt{|V''_{0, \delta}
			(y_{m_{i-1}})|} + \sqrt{|V''_{0, \delta} (y_{m_i})|}}. 
\end{eqnarray}
Since $\psi (x) = P_{x}\big( X^\varepsilon_\delta (\hat \tau^{\varepsilon, \delta}_i) = y_{m_i}\big), x \in W_i$ uniquely solve 
\begin{equation}\label{generatorYepsilondelta}
	L^{\varepsilon, \delta} \psi := \frac{\varepsilon^2}{2} \psi'' 
	- V'_{\varepsilon, \delta} \psi' = 0, \  \psi (y_{m_{i-1}} ) = 0, \psi (y_{m_i}) = 1.
\end{equation}
Hence, we get
\[
p^{\varepsilon, \delta}_{i i +1} \ = \ \frac{\int^{x_{m_i}}_{y_{m_{i-1}}} 
	e^{\frac{2}{\varepsilon^2} V_{\varepsilon, \delta}(x)} dx}{\int^{y_{m_i}}_{y_{m_{i-1}}} 
	e^{\frac{2}{\varepsilon^2} V_{\varepsilon, \delta}(x)} dx}.
\]
Now repeating similar compuations as in the proof of Lemma \ref{normalizingconstant}, we have
the following.
\begin{lemma}\label{exittimeestimate-new} Follwing estimates hold.
	\begin{eqnarray*}
		p^{\varepsilon, \delta}_{i i +1} & = & p^\delta_{i i + 1} + O(\varepsilon), \\
		p^{\varepsilon, \delta}_{i i -1}& = & p^\delta_{i i - 1} + O(\varepsilon). 
	\end{eqnarray*}
\end{lemma}
Set
\begin{eqnarray}\label{Qmatrix}
	\mu^\delta (x_{m_i}) & = & \frac{1}{C_\delta \sqrt{ | V''_{0, \delta} (x_{m_i})|}}, 
	\nonumber  \\ \nonumber
	\lambda (1) & = & \frac{1}{C_\delta \mu^\delta (x_{m_1}) \sqrt{ | V''_{0, \delta} (y_{m_1})|}}  \\ 
	\lambda (i) & = & \frac{1}{C_\delta \mu^\delta (x_{m_i})} 
	\Big(\frac{1}{\sqrt{ | V''_{0, \delta} (y_{m_{i-1}})|}} + 
	\frac{1}{\sqrt{ | V''_{0, \delta} (y_{m_i})|}} \Big), i \geq 2, \\ \nonumber
	Q^\delta (x_{m_i}, x_{m_{i +1}}) & = & \frac{1}{Z^\delta \mu^\delta(x_{m_i}) 
		\sqrt{ | V''_{0, \delta} (y_{m_i})|}}, \, i =1, 2, \cdots, \kappa -1,\\  \nonumber
	Q^\delta (x_{m_i}, x_{m_{i -1}}) & = & \frac{1}{Z^\delta \mu^\delta(x_{m_i}) 
		\sqrt{ | V''_{0, \delta} (y_{m_{i-1}})|}}, \, i = 2, \cdots, \kappa.
\end{eqnarray}
Let $Z_\delta (\cdot)$ denote the continuous time Markov chain with state space $S$ with rate 
matrix $Q^\delta$. Then its generator is given by, for $ i = 1, 2, \cdots , \kappa$, 
\begin{equation}\label{generatorYdelta}
	L^\delta f(x_{m_i}) = \lambda (i) p^\delta_{i i-1} \big(f(x_{m_i}) - f(x_{m_{i-1}}) \big)
	+ \lambda (i) p^\delta_{i i+1} \big(f(x_{m_i}) - f(x_{m_{i+1}}) \big).
\end{equation}

Let $h_i : \mathbb{R} \to [0, \ 1]$ be a smooth function such that 
$ h_i (x) = 1, x \in V_i, \ h_i(x) = 0, x \in \mathbb{R} \setminus V'_i, 
i =1, 2, \cdots, \kappa $. 
\begin{lemma}\label{blockodeestimate} For $F : S \to \mathbb{R}$, the  ode 
	\begin{equation}\label{blockode}
		e^{\frac{2}{\varepsilon^2} \lambda^\delta} L^{\varepsilon, \delta } f_\varepsilon = 
		- \sum^\kappa_{i=1} L^\delta  F (x_{m_i})h_i 
	\end{equation}
	has a solution $f_\varepsilon \in W^{2,p}(\mathbb{R}), \, p \geq 2$ such that 
	$f_\varepsilon  (x) \to  F (x_{m_i}), x \in V_i, i = 1, 2, \cdots, \kappa$. 
\end{lemma}
\begin{proof} Let $V_i = ( y^1_i, \ y^2_i), \, i =1, 2, \cdots, \kappa$. 
	Let $(a(i), b(i)) , \, i =1,2 \cdots , \kappa$ are chosen such that
	\begin{eqnarray}
		F(x_{m_1}) & = &  LF (x_{m_1}) + b(1) p^\delta_{1 2} \nonumber \\
		F(x_{m_i}) & = & LF (x_{m_i})  
		+ a(i) p^{\delta}_{i i-1} + b(i) p^{ \delta}_{i i+1} , i =2, \cdots, \kappa-1,\\ \nonumber
		F(x_{m_\kappa}) & = &  LF (x_{m_\kappa})  
		+ a(\kappa) p^\delta_{\kappa \kappa -1} . 
	\end{eqnarray}
	Observe that the choice is not unique. Set $a(1) = b(\kappa) =0$. 
	Consider the ode 
	\begin{eqnarray}\label{Poissoneq1}
		L^{\varepsilon, \delta} f & = & - e^{-\frac{2}{\varepsilon^2} \lambda^\delta} \, 
		L^\delta F(x_{m_i}) , \ x \in V_i, \nonumber \\ 
		f(y^1_i) & = &  a(i), f(y^2_i) = b(i), \ i =1, 2, \cdots , \kappa.
	\end{eqnarray}
	Let $\tilde \tau^{\varepsilon, \delta}_i$ denote the exit time of 
	$X^\varepsilon_\delta (\cdot)$ from $V_i, i = 1, 2, \cdots, \kappa.$
	Then it is easy to see that solution $f^i_\varepsilon \in C^2 (V_i),$ of (\ref{Poissoneq1})
	satisfies
	\begin{eqnarray*}
		f^1_\varepsilon (x) & = & e^{-\frac{2}{\varepsilon^2} \lambda^\delta}
		LF (x_{m_1}) E_x [ \tilde \tau^{\varepsilon, \delta}_1 ] 
		+ b(1) p^\delta_{1 2} + O(\varepsilon), x \in V_1 \\
		f^i_\varepsilon (x) & = & e^{-\frac{2}{\varepsilon^2} \lambda^\delta} LF (x_{m_i}) 
		E_x [ \tilde \tau^{\varepsilon, \delta}_i ] 
		+ a(i) p^{\delta}_{i i-1} + b(i) p^{ \delta}_{i i+1} + O(\varepsilon) , x \in V_i , \\
		&&  i =2, \cdots, \kappa-1,\\
		f^\kappa_\varepsilon (x) & = &  e^{-\frac{2}{\varepsilon^2} \lambda^\delta}
		LF (x_{m_\kappa}) E_x [ \tilde \tau^{\varepsilon, \delta}_1 ] 
		+ a(\kappa) p^\delta_{\kappa \kappa -1} + O(\varepsilon), x \in V_\kappa . 
	\end{eqnarray*}
	Consider the ode 
	\begin{eqnarray}\label{Poissoneq2}
		L^{\varepsilon, \delta} g & = & - e^{-\frac{2}{\varepsilon^2} \lambda^\delta} \, 
		L^\delta F(x_{m_i})h_i , \ x \in W_i \setminus V_i, \nonumber \\ 
		g(y^2_i) & = &  f^i_\varepsilon(y^2_i), g(y^1_i) = f^i_\varepsilon(y^2_i), \\ \nonumber
		g(y) & = & W_i \setminus V'_i, \ \ i =1, 2, \cdots , \kappa .
	\end{eqnarray}
	Then there exists unique solution $g^i_\varepsilon \in C^2 (W_i \setminus V_i)$ to 
	(\ref{Poissoneq2}). Then  
	\[
	f_\varepsilon = \sum^\kappa_{i =1} \big( f^i_\varepsilon I_{V_i} + g^i_\varepsilon 
	I_{W_i \setminus V_i} \big).
	\]
	is a  weak solution to (\ref{blockode}) in $W^{1, p}(\mathbb{R}), p \geq 2$. 
	A straightforward regularity argument implies that 
	$f_\varepsilon \in W^{2, p}(\mathbb{R}), p \geq 2$ and hence an a.e. solution as well.
	From Lemma \ref{exittimeestimate} and the definition of $\lambda^\delta$, we have 
	\[
	\lim_{\varepsilon \to 0} \varepsilon^2 \ln E_x \tilde \tau^{\varepsilon, \delta}_i \ = \
	\lambda^\delta
	\]
	Hence $f_\varepsilon (x) \to  \sum^\kappa_{i=1}  F (x_{m_i}), \, x \in V_i$ pointwise.

\end{proof}
Now by mimicking the arguments from \cite{Seo}, we have the following result.
\begin{lemma}\label{auxillaryestimates} (i) For each $t > 0$,
	\[
	\lim_{\varepsilon \to 0} \sup_{x \in V_i} P_x \big( \ \breve \tau^{\varepsilon, \delta}_i \leq t
	\big) \leq K_1 t ,
	\]
	for some $K_1 > 0$, where  $\breve \tau^{\varepsilon, \delta}_i$ denote the exit time for 
	$\hat X^\varepsilon_\delta (\cdot)$ from $W_i$. 
	
	(ii) For $x \in V$, 
	\[
	\lim_{\varepsilon \to 0} E_x \Big[ \int^t_0 I_{\{ \hat X^\varepsilon_\delta (s) \in V^c \}} ds
	\Big] \ = \ 0, \, t > 0 .
	\]
	
\end{lemma}

Following the arguments of \cite{Seo}, using Lemma \ref{auxillaryestimates}, we have 
\begin{lemma}\label{tightness}
	For each $\delta >0$, the laws of the process $Z^\varepsilon_\delta (\cdot)$ is tight.
\end{lemma}

\begin{theorem}\label{tunneling} For each $\delta > 0$, the process 
	$Z^\varepsilon_\delta (\cdot)$ converges in law to
	the continuous time Markov chain $Z_\delta (\cdot)$ as $\varepsilon \to 0$.
\end{theorem}

\begin{proof} Consider the process $Z^\varepsilon_\delta (\cdot)$. Let $Z(\cdot)$ be a limit point, we take $\varepsilon \to 0$ as the corresponding subsequential limit for simplicity. 
	Set
	\[
	g = \sum^\kappa_{i=1} L^\delta  F (x_{m_i})h_i .
	\]
	Consider the martingale
	\begin{eqnarray*}
		M^\varepsilon_\delta (t) & = & f_\varepsilon (\hat X^\varepsilon_\delta (t)) - \int^t_0 \hat L^{\varepsilon, \delta} f_\varepsilon (\hat X^\varepsilon_\delta (s)) ds\\
		& = & f_\varepsilon (\hat X^\varepsilon_\delta (t))  - \int^t_0 g (\hat X^\varepsilon_\delta (s)) ds\\
		& = & f_\varepsilon (\hat X^\varepsilon_\delta (t))  - \int^t_0 g (\hat X^\varepsilon_\delta (s)) I_{\{\hat X^\varepsilon_\delta (s) \in V \}}  ds
		+ \int^t_0 g (\hat X^\varepsilon_\delta (s)) I_{\{\hat X^\varepsilon_\delta (s) \in V^c \}}  ds.
	\end{eqnarray*}
	Since $S^{\varepsilon, \delta}(t)$ is a stopping time with respect to $\{{\mathcal F}_t\}$, it follows that
	\[
	\hat M^\varepsilon_\delta (t) = M^\varepsilon_\delta (S^{\varepsilon, \delta} (t)), t \geq 0
	\]
	is a $\{{\mathcal G}_t\}$-martingale. 
	Now from the definition of $S^{\varepsilon, \delta}(\cdot)$
	and $h_i$, it follows that
	\[
	\hat M^\varepsilon_\delta (t) \ = \ f_\varepsilon (Z^\varepsilon_\delta (t)) - 
	\sum^\kappa_{i=1} \int^t_0 L^\delta  F (x_{m_i}) I_{V_i} (Z^\varepsilon_\delta (s)) ds 
	+ O(\varepsilon) . 
	\]
	For $ 0 \leq t_1 < t_2  \cdots < t_n \leq s < t, n \geq 1, \varphi \in C_b (\mathbb{R}^{n})$,
	we have
	\[
	E \Big[ \varphi( Z^\varepsilon_\delta (t_1), Z^\varepsilon_\delta (t_2), \cdots , 
	Z^\varepsilon_\delta (t_n) ) 
	\big(\hat M^\varepsilon_\delta (t) - \hat M^\varepsilon_\delta (s) \big) \Big] \ = \ 0
	\]
	Now by letting $\varepsilon \to 0$, using Lemma \ref{blockodeestimate},  we get
	\[
	E \Big[ \varphi( Z (t_1), Z (t_2), \cdots , Z (t_n) ) 
	\big(\hat M (t) - \hat M (s) \big) \Big] \ = \ 0,
	\]
	where 
	\[
	\hat M(t) = F(Z(t) ) - \int^t_0 L^\delta F(Z(s)) ds , t \geq 0.
	\]
	Therefore $\hat M (\cdot)$ is a martigale. 
	Hence $Z(\cdot)$ is a continuous time Markov chain with generator $L^\delta$.
	This completes the proof.
	
\end{proof}
Let $Z(\cdot)$ denote the continuous time Markov chain with state space $S$ and generator 
$L$ given by
\begin{equation}\label{generatorL}
	L f(x_{m_i}) \ = \ 
	\lambda (i)  p_{i i-1} \big(f(x_{m_i}) -f(x_{m_{i-1}}) \big) +  \lambda (i)  p_{i i+1}
	\big(f(x_{m_i}) - f(x_{m_{i+1}}) \big),
\end{equation}
where 
\begin{eqnarray}\label{mu}
	p_{i i +1} & = & \frac{\sqrt{|V''_{0} (y_{m_{i-1}+})|}}{\sqrt{|V''_{0}
			(y_{m_{i-1}}+)|} + \sqrt{|V''_{0} (y_{m_i}+)|}}  \nonumber \\ \nonumber
	p_{i i -1} & = & \frac{\sqrt{|V''_{0} (y_{m_i}+)|}}{\sqrt{|V''_{0}
			(y_{m_{i-1}}+)|} + \sqrt{|V''_{0} (y_{m_i}+)|}},  \\ \nonumber
	C_0 & = & \sum^\kappa_{j=1} \frac{1}{\sqrt{ | V''_{0} (x_{m_j}+)|}} \\ 
	\mu (x_{m_i}) & = & \frac{1}{ C_0 \sqrt{ | V''_{0} (x_{m_i}+)|}},  \\ \nonumber
	\lambda (1) & = & \frac{1}{ C_0  \mu (x_{m_1}) \sqrt{ | V''_{0}(y_{m_1}+)|}}  \\  \nonumber
	\lambda (i) & = & \frac{1}{ C_0  \mu^ (x_{m_i})} 
	\Big(\frac{1}{\sqrt{ | V''_{ 0} (y_{m_{i-1}}+)|}} + 
	\frac{1}{\sqrt{ | V''_{0} (y_{m_i}+)|}} \Big), i \geq 2.
\end{eqnarray}

\begin{lemma}\label{invariantprobability} The invariant distribution 
	$\eta^{\varepsilon, \delta}$ converges to the unique invariant ditsribution of 
	$Z (\cdot)$ as $\varepsilon \to 0, \delta \to 0$ in that order.
\end{lemma}
\begin{proof}This follows from Lemma \ref{normalizingconstant} (i) and the observation that
	$\mu^\delta $ given in (\ref{Qmatrix}) is the unique invariant distribution of 
	$Z^\delta (\cdot)$. Now from (\ref{generatorL}), it follows that $\mu^\delta \to \mu$,
	the unique invariant measure of $Z(\cdot)$.  
\end{proof}
\begin{lemma}\label{H2} Assume (T). 
	Then assumption (H2) holds for $u^0$.
	
\end{lemma}
\begin{proof}
	Let $u^\delta(\cdot)$ be a $\delta$-optimal smooth stationary Markov control 
	for (\ref{stateode-tunneling}); (\ref{cost-tunneling}), i.e.,
	there exists $\pi^\delta \in {\mathcal G}$ such that 
	\begin{eqnarray}\label{deltau}
		\iint \bar{r} (x, u) \pi^\delta (dxdu) & \leq & \rho^* + \delta , \nonumber \\
		\iint f(x, u) \pi^\delta (dx du) & = & \int f(x, u^\delta (x)) \eta^\delta (dx),
	\end{eqnarray}
	for some $\eta^\delta \in {\mathcal P}(\mathbb{R})$.
	Set
	\begin{equation}\label{deltapotential}
		V_\delta (x) = \ V(x) - \int^x_0 u^\delta (y) dy, x \in \mathbb{R}
	\end{equation}
	
	Let $X^{\varepsilon, \delta} (\cdot)$
	and $X^\delta (\cdot)$ denote respectively the solution to (\ref{statesde-tunneling}) and
	(\ref{stateode-tunneling}) corresponding to $u^\delta (\cdot)$. 
	Let $\hat \eta^{\varepsilon, \delta} (dx)$ dnote the unique invariant probability measure
	of $X^{\varepsilon, \delta}(\cdot)$. 
	Set
	\begin{eqnarray}
		\hat p^\delta_{i i +1} & = & \frac{\sqrt{|V''_{ \delta} (y_{m_{i-1}})|}}{\sqrt{|V''_{\delta}
				(y_{m_{i-1}})|} + \sqrt{|V''_{ \delta} (y_{m_i})|}} \nonumber \\
		\hat p^\delta_{i i -1} & = & \frac{\sqrt{|V''_{\delta} (y_{m_i})|}}{\sqrt{|V''_{\delta}
				(y_{m_{i-1}})|} + \sqrt{|V''_{\delta} (y_{m_i})|}},  \label{hattransitionprob} \\ \nonumber 
		\hat C_\delta & = & \sum^\kappa_{j=1} \frac{1}{\sqrt{ | V''_{ \delta} (x_{m_j})|}}, \\ \nonumber 
		\hat \mu^\delta (x_{m_i}) & = & \frac{1}{\hat C_\delta \sqrt{ | V''_{\delta} (x_{m_i})|}}, 
		\nonumber  \\ \nonumber
		\hat \lambda (1) & = & \frac{1}{\hat C_\delta \hat \mu^\delta (x_{m_1}) 
			\sqrt{ | V''_{ \delta}(y_{m_1})|}},  \\ \nonumber
		\hat \lambda (i) & = & \frac{1}{\hat C_\delta \hat \mu^\delta (x_{m_i})} 
		\Big(\frac{1}{\sqrt{ | V''_{\delta} (y_{m_{i-1}})|}} + 
		\frac{1}{\sqrt{ | V''_{\delta} (y_{m_i})|}} \Big), i \geq 2, \\ 
		\hat Q^\delta (x_{m_i}, x_{m_{i +1}}) & = & \frac{1}{\hat C_\delta \hat \mu^\delta(x_{m_i}) 
			\sqrt{ | V''_{\delta} (y_{m_i})|}}, \, i =1, 2, \cdots, \kappa -1, 
		\label{hattransitionrate} \\  \nonumber
		\hat Q^\delta (x_{m_i}, x_{m_{i -1}}) & = & \frac{1}{\hat C_\delta \hat \mu^\delta(x_{m_i}) 
			\sqrt{ | V''_{ \delta} (y_{m_{i-1}})|}}, \, i = 2, \cdots, \kappa. 
	\end{eqnarray}
	Let $\hat Z_\delta (\cdot)$ denote the continuous time Markov chain with state space $S$
	with rate  matrix $\hat Q^\delta$. 
	Then the generator is given by for $, i = 1, 2, \cdots , \kappa$, 
	\begin{equation}\label{generatorhatYdelta}
		\hat L^\delta f(x_{m_i}) = \hat \lambda (i) \hat p^\delta_{i i-1} \big(f(x_{m_i}) -
		f(x_{m_{i-1}}) \big) + \hat \lambda (i) \hat p^\delta_{i i+1}
		\big(f(x_{m_i}) - f(x_{m_{i+1}}) \big).
	\end{equation}
	Now by mimicking the  analysis used to prove Theorem \ref{tunneling}, it follows that 
	$\hat \eta^{\varepsilon, \delta} (dx)$
	converges to the unique invariant probability distribution $\hat \mu^\delta$ of  
	$\hat Z_\delta (\cdot)$. Clearly $\hat \mu^\delta $ is an invariant probability distribution
	of $\hat Z_\delta (\cdot)$. 
	Now 
	\[
	\int \bar{r}(x, u^\delta (x)) \hat \mu^\delta (dx) \ \leq \ \rho^* + \delta.
	\]
	Hence (H2) holds for $u^0$.

\end{proof} 
\begin{theorem}\label{tunnelvalue} Assume (T). Then 
	\[
	\lim_{\delta \to 0} \lim_{\varepsilon \to 0} 
	\rho_\varepsilon (x, u^\varepsilon_\delta (\cdot) =  \rho^* .
	\]
	More over $\rho^*$ is given by
	\[
	\rho^* = \lim_{t \to \infty} \frac{1}{t} E \Big[\int^t_0 r(Z(s), u^0(Z(s)) ds \Big].
	\] 
\end{theorem}

\begin{proof} 
	Consider
	\begin{eqnarray*}
		\rho_\varepsilon(x, u^\varepsilon_\delta (\cdot) & = & 
		\liminf_{t \to \infty} \frac{1}{t} E\Big[ \int^t_0 
		r(X^\varepsilon_\delta (s), u^\varepsilon_\delta (X^\varepsilon_\delta (s)) ds \\
		& = & \int r (x, u^\varepsilon_\delta (x) ) \eta^{\varepsilon, \delta} (dx) \\
		& = & \int (r (x, u^\varepsilon_\delta (x) - r (x, u^0_\delta (x) )
		\eta^{\varepsilon, \delta} (dx) + 
		\int \bar{r} (x, u^0_\delta (x) ) \eta^{\varepsilon, \delta} (dx)
	\end{eqnarray*}
	Since $u^\varepsilon_\delta \to u^0_\delta $ uniformly on compact sets, it follows from
	Lemma \ref{invariantprobability} 
	\[
	\lim_{\varepsilon \to 0} \rho_\varepsilon(x, u^\varepsilon_\delta (\cdot)) = 
	\int r(x, u^0_\delta (x)) \mu^\delta (dx) . 
	\]
	Hence using (\ref{approximatevalue-eps}), we get
	\begin{eqnarray*}
		\lim_{\varepsilon \to 0} \rho_\varepsilon(x, u^\varepsilon (\cdot)) & = & 
		\lim_{\delta \to 0} \int r(x, u^0_\delta (x)) \mu^\delta (dx) \\
		& = & \int r(x, u^0 (x)) \mu (dx), 
	\end{eqnarray*}
	where $\mu(dx)$ is the unique invariant probability measure of $Z(\cdot)$ given in (\ref{mu}).
	The last equality follows since $\mu^\delta, \mu $ are supported on $S$.
	From Theorem \ref{metatheorem} and Lemma \ref{H2}, we have 
	\[
	\lim_{\varepsilon \to 0}  \rho_\varepsilon (x, u^\varepsilon (\cdot)) \, = \, 
	\rho^*.
	\]
	Hence we get
	\begin{equation}\label{repvalue}
		\rho^* = \int r(x, u^0 (x)) \mu (dx).
	\end{equation}
	
	This completes the proof.
	
\end{proof}

\begin{appendix}
	\section*{}
	\begin{lemma}\label{weakFeller} Assume (A1). For any prescribed control $U(\cdot)$, let $X(\cdot)$ denote a unique solution to the sde (\ref{statedegensde}). 
		The semigroup $\{T_t : t \geq 0 \}$ of $X(\cdot)$ satisfies the following. For each $f \in C^2_b(\mathbb{R}^d),$ $T_t f$ is 
		Lipschitz continuous for each $t \geq 0$. 
	\end{lemma}
	\begin{proof} For $x_1, x_2 \in \mathbb{R}^d$, let $X(x_1; \cdot), X(x_2; \cdot)$ denote unique solutions
		to (\ref{statedegensde}) corresponding to $U(\cdot)$ and initial conditions $x_1, x_2$ respectively.
		Now using Ito's formula we get for $f \in C^2_b(\mathbb{R}^d)$, 
		\begin{equation}\label{WeakFellereq1}
			d f(X(x_i; t)) \ = \ {\mathcal L}^{U(t)} f(X(x_i; t)) dt 
			+ \nabla f(X(x_i; t)) \sigma (X(x_i; t)) d W(t), i = 1, 2.
		\end{equation} 
		Using (A1) and (\ref{WeakFellereq1}),  we have
		\begin{eqnarray*}
			| T_t f(x_1) - T_t f(x_2) | & = & | E f(X(x_1; t)) - E f(X(x_2; t)) | \\
			& \leq & | f(x_1) - f(x_2) + K_1 \int^t_0 E \| X(x_1; s) - X(x_2; s) \| ds \\
			& \leq & K_2 \| x_1 - x_2\|,
		\end{eqnarray*}
		where the last inequality follows from Remark 2.2.6, p.39, \cite{AriBorkarGhosh} and $K_1, K_2$  
		are constants which may depend on $T, f$ but not on $x_1, x_2$.  
		
	\end{proof}  
	
	\begin{lemma}\label{invariantmeasurechar1} Assume (A1). Consider the process
		$Z^\varepsilon(\cdot )$ given in Lemma \ref{asymtoticbehavior}. 
		If $\mu^\varepsilon \in \mathcal{P}(\mathbb{R}^{3d})$
		satisfies 
		\[
		\iiint {\mathcal L}^{\varepsilon}_Z f (z) \mu^\varepsilon (dz) = 0 \ {\rm for\ all} \ f \in C^\infty_c (\mathbb{R}^{3 d}),
		\]
		then $\mu^\varepsilon$ is an invariant probability measure of the process $Z^\varepsilon (\cdot)$.
	\end{lemma} 
	\begin{proof} For $f \in C^\infty_c (\mathbb{R}^{3d})$,  using 
		It$\hat {\rm o}$'s formula), we have
		\[
		E_z f (Z^\varepsilon_t) = f(z) + \int^t_0 E_z [ {\mathcal L}^\varepsilon_Z f(Z^\varepsilon_s)] ds 
		\]
		Define 
		\[
		T_t f (z) = E_z f(Z^\varepsilon_t),  z \in \mathbb{R}^{3d}, \, t > 0, T_0 f = f .
		\]
		Then $\{T_t : t \geq 0\}$ is a strongly continuous semigroup with generator
		${\mathcal L}^\varepsilon_Z$.  Then 
		for $ f \in {\mathcal D}( {\mathcal L}^\varepsilon_Z)$, from Propostion 1.3.1, p.10,  \cite{AriBorkarGhosh}, we have 
		\[
		T_t f \in {\mathcal D}( {\mathcal L}^\varepsilon_Z), \ 
		{\mathcal L}^\varepsilon_Z (T_t f ) = T_t ( {\mathcal L}^\varepsilon_Z) 
		\]
		and 
		\[
		E_z f (X_t) \ =  \  f(z) + \int^t_0 {\mathcal L}^\varepsilon_Z (T_s f) (z) ds  . 
		\]
		Integrate the above with respect to $\mu^\varepsilon$, we get
		\begin{eqnarray*}
			\iiint E_z f (Z^\varepsilon_t) \mu^\varepsilon (dz) &  =  & \iiint f(z) \mu^\varepsilon (dz)
			+ \int^t_0 \iiint  {\mathcal L}^\varepsilon_Z (T_s f) (z) \mu^\varepsilon (dz) ds \\
			& = & \iiint f(z) \mu^\varepsilon (dz). 
		\end{eqnarray*}
		The last equality holds, since $T_s f \in {\mathcal D}( {\mathcal L}^\varepsilon_Z)$.
		In particular, since 
		$C^\infty_c(\mathbb{R}^{3d})\subseteq {\mathcal D}( {\mathcal L}^\varepsilon_Z)$, we have
		\[
		\iiint E_z f (Z^\varepsilon_t) \mu^\varepsilon (dz) \ = \ \iiint f(z) \mu^\varepsilon (dz)
		\, {\rm for\ all} \ f \in C^\infty_c(\mathbb{R}^{3d}), \, t > 0.
		\]
		Hence  $\mu^\varepsilon$ is  an invariant probability measure for $Z^\varepsilon(\cdot)$. 
	\end{proof}
	
	\begin{remark}The above result doesn't assume that process $Z^\varepsilon (\cdot)$ is Feller
		instead we assume 
		$C^\infty_c(\mathbb{R}^{3d})\subseteq {\mathcal D}( {\mathcal L}^\varepsilon_Z)$ but the proof is standard. We give it for the sake completeness.
	\end{remark}
	
\end{appendix}

\end{document}